\documentclass[11pt]{article}
\usepackage[width=15cm,height=21cm]{geometry}
\usepackage[colorlinks,allcolors=blue]{hyperref}
\usepackage{comment}

\newcommand{\doi}[1]{\href{https://doi.org/#1}{doi:#1}}

\usepackage{cite}
\newcommand{\myurl}[1]{\href{#1}{#1}}
\newcommand{\journaltitle}[1]{#1}
\newcommand{\booktitle}[1]{#1}

\newcommand{\volumeyearpages}[3]{\textbf{#1}, #3 (#2)}
\newcommand{\volumeissueyearpages}[4]{\textbf{#1}:#2, #4 (#3)}

\usepackage{amsmath,amssymb}
\usepackage{mathtools}
\newcommand{\eqdef}{\coloneqq}

\newcommand{\bC}{\mathbb{C}}
\newcommand{\bD}{\mathbb{D}}
\newcommand{\bZ}{\mathbb{Z}}
\newcommand{\bN}{\mathbb{N}}
\newcommand{\bNz}{\mathbb{N}_0}
\newcommand{\bR}{\mathbb{R}}
\newcommand{\bT}{\mathbb{T}}
\newcommand{\cA}{\mathcal{A}}
\newcommand{\cB}{\mathcal{B}}

\newcommand{\cP}{\mathcal{P}}
\newcommand{\cR}{\mathcal{R}}

\newcommand{\cU}{\mathcal{U}}

\newcommand{\fM}{\mathfrak{M}}

\newcommand{\econstant}{\operatorname{e}}
\newcommand{\imagunit}{\operatorname{i}}
\newcommand{\bcoef}[3]{\widetilde{c}^{(#1)}_{#2,#3}}
\newcommand{\basic}[3]{b^{(#1)}_{#2,#3}}
\newcommand{\al}{\alpha}

\newcommand{\Be}{\operatorname{B}}
\newcommand{\be}{\beta}
\newcommand{\Ga}{\Gamma}
\newcommand{\ga}{\gamma}
\newcommand{\de}{\delta}

\newcommand{\dif}{\mathrm{d}}

\newcommand{\la}{\lambda}

\newcommand{\Om}{\Omega}
\newcommand{\Mat}{\mathcal{M}}

\newcommand{\medstrut}{\vphantom{0_{0_0}^{0^0}}}
\newcommand{\clos}{\operatorname{clos}}
\newcommand{\lin}{\operatorname{span}}

\newcommand{\conjw}{\overline{w}}
\newcommand{\conjz}{\overline{z}}
\newcommand{\Berezin}{\operatorname{Ber}}
\newcommand{\Radialization}{\operatorname{Rad}}
\newcommand{\radialization}{\operatorname{rad}}

\newcommand{\jac}{\mathcal{J}}
\newcommand{\jaccoef}[3]{c^{(#1,#2)}_{#3}}
\newcommand{\sign}{\operatorname{sgn}}

\newcommand{\FreqSubspace}{\mathcal{W}}

\usepackage{colortbl}
\usepackage{xcolor}
\colorlet{lightblue}{blue!20}

\usepackage{amsthm}
\newtheorem{theorem}{Theorem}[section]
\newtheorem{proposition}[theorem]{Proposition}
\newtheorem{lemma}[theorem]{Lemma}
\newtheorem{corollary}[theorem]{Corollary}
\theoremstyle{definition}
\newtheorem{example}[theorem]{Example}
\newtheorem{definition}[theorem]{Definition}
\newtheorem{remark}[theorem]{Remark}

\author{Roberto Mois\'{e}s Barrera-Castel\'{a}n,
Egor A. Maximenko,\\[1ex]
Gerardo Ramos-Vazquez}
\title{Radial operators on polyanalytic\\weighted Bergman spaces}

\begin{document}%
\maketitle

\begin{abstract}
Let $\mu_\alpha$ be the Lebesgue plane measure on the unit disk with the radial weight $\frac{\alpha+1}{\pi}(1-|z|^2)^\alpha$. Denote by $\mathcal{A}^{2}_{n}$ the space of the $n$-analytic functions on the unit disk, square-integrable with respect to $\mu_\alpha$. Extending the results of Ramazanov (1999, 2002), we explain that disk polynomials (studied by Koornwinder in 1975 and W\"{u}nsche in 2005) form an orthonormal basis of $\mathcal{A}^{2}_{n}$. Using this basis, we provide the Fourier decomposition of $\mathcal{A}^{2}_{n}$ into the orthogonal sum of the subspaces associated with different frequencies. This leads to the decomposition of the von Neumann algebra of radial operators, acting in $\mathcal{A}^{2}_n$, into the direct sum of some matrix algebras. In other words, all radial operators are represented as matrix sequences. In particular, we represent in this form the Toeplitz operators with bounded radial symbols, acting in $\mathcal{A}^{2}_n$. Moreover, using ideas by Engli\v{s} (1996), we show that the set of all Toeplitz operators with bounded generating symbols is not weakly dense in $\mathcal{B}(\mathcal{A}^{2}_n)$.


\end{abstract}

\section{Introduction}

\subsection*{Background}

Polyanalytic functions naturally arise in some physical models (plane elasticity theory, Landau levels) and in some methods of signal processing,
see~\cite{AbreuFeichtinger2014,HaimiHedenmalm2013,AliBagarelloGazeau2015,Abreu2010,Hutnik2008,Hutnik2010}.
Koshelev~\cite{Koshelev1977} computed the reproducing kernel of the $n$-analytic Bergman space $\cA_n^2(\bD)$ on the unit disk.
Balk in the book~\cite{Balk1991} explained
fundamental properties of polyanalytic functions.
Dzhuraev~\cite{Dzhuraev1992}
related polyanalytic projections
with singular integral operators.
Vasilevski~\cite{Vasilevski1999,Vasilevski2000}
studied polyanalytic Bergman spaces
on the upper halfplane
and polyanalytic Fock spaces
using the Fourier transform.
Ramazanov~\cite{Ramazanov1999,Ramazanov2002}
constructed an orthonormal basis in $\cA_n^2(\bD)$ and studied various properties of $\cA_n^2(\bD)$.
In fact, the elements of this basis are well-known disk polynomials
studied by Koornwinder~\cite{Koornwinder1975},
W\"{u}nsche~\cite{Wunsche2005},
and other authors.
Pessoa~\cite{Pessoa2014} related Koshelev's formula with Jacobi polynomials and gave a very clear proof of this formula.
He also obtained similar results for some other one-dimensional domains.
Hachadi and Youssfi~\cite{HachadiYoussfi2019}
developed a general scheme for computing
the reproducing kernels of the spaces of polyanalytic functions on radial plane domains
(disks or the whole plane) with radial measures.

There are general investigations
about bounded linear operators
in reproducing kernel Hilbert spaces (RKHS),
especially about Toeplitz operators
in Bergman or Fock spaces \cite{Zhu2007book, BauerFulsche2020,Xia2015}, 
but the complete description of the spectral properties is found
only for some special classes of operators,
in particular, for Toeplitz operators
with generating symbols
invariant under some group actions,
see Vasilevski~\cite{Vasilevski2008book},
Grudsky, Quiroga-Barranco, and Vasilevski~\cite{GrudskyQuirogaVasilevski2006},
Dawson, \'{O}lafsson, and Quiroga-Barranco~\cite{DawsonOlafssonQuiroga2015}.
The simplest class of this type consists of Toeplitz operators
with bounded radial generating symbols.
Various properties of these operators
(boundedness, compactness, and eigenvalues)
have been studied by many authors,
see~\cite{KorenblumZhu1995,GrudskyVasilevski2002,Zorboska2003,Quiroga2016}.
The C*-algebra generated by such operators,
acting in the Bergman space,
was explicitly described in \cite{Suarez2008,GrudskyMaximenkoVasilevski2013,BauerHerreraVasilevski2014,HerreraVasilevskiMaximenko2015}.
Loaiza and Lozano \cite{LoaizaLozano2014}
obtained similar results for
radial Toeplitz operators in harmonic Bergman spaces.
Maximenko and Teller\'{i}a-Romero~\cite{MaximenkoTelleriaRomero2020} studied radial operators in the polyanalytic Fock space.

Hutn\'{i}k, Hutn\'{i}kov\'{a}, Ram\'{i}rez-Mora, Ram\'{i}rez-Ortega, S\'{a}nchez-Nungaray, Loaiza, and other authors
\cite{HutnikHutnikova2015,RamirezSanchez2015,LoaizaRamirez2017,HutnikMaximenkoMiskova2016,RamirezRamirezSanchez2019,SanchezGonzalezLopezArroyo2018}
studied vertical and angular Toeplitz operators in polyanalytic and true-polyanalytic Bergman spaces.
In particular, vertical Toeplitz operators in the
$n$-analytic Bergman space over the upper half-plane
are represented in \cite{RamirezSanchez2015}
as $n\times n$ matrices
whose entries are continuous functions on $(0,+\infty)$,
with some additional properties at $0$ and $+\infty$.

Rozenblum and Vasilevski
\cite{RozenblumVasilevski2019}
investigated Toeplitz operators with distributional symbols
and showed that Toeplitz operators
in true-polyanalytic spaces Bergman or Fock spaces are equivalent to some Toeplitz operators with distributional symbols in the analytic Bergman or Fock spaces.

\subsection*{Objects of study}

Denote by $\mu$ the Lebesgue plane measure
and its restriction to the unit disk $\bD$,
and by $\mu_\alpha$ the weighted Lebesgue plane measure
\[
\dif\mu_\al(z)\eqdef
\frac{\al+1}{\pi}\,(1-|z|^2)^\al\,\dif\mu(z).
\]
This measure is normalized:
$\mu_\al(\bD)=1$.
We use notation $\langle\cdot,\cdot\rangle$
and $\|\cdot\|$ for the inner product
and the norm in $L^2(\bD,\mu_\al)$.

Let $\cA_n^2(\bD,\mu_\al)$ be the space
of $n$-analytic functions square integrable
with respect to $\mu_\al$.
We denote by $\cA_{(n)}^2(\bD,\mu_\al)$
the orthogonal complement of $\cA_{n-1}^2(\bD,\mu_\al)$
in $\cA_n^2(\bD,\mu_\al)$.

For every $\tau$ in the unit circle
$\bT\eqdef\{z\in\bC\colon\ |z|=1\}$,
let $\rho^{(\al)}_n(\tau)$ be the rotation
operator acting in $\cA_n^2(\bD,\mu_\al)$ by the rule
\[
(\rho^{(\al)}_n(\tau)f)(z)
\eqdef f(\tau^{-1}z).
\]
The family $\rho^{(\al)}_n$
is a unitary representation of the group $\bT$
in the space $\cA_n^2(\bD,\mu_\al)$.
We denote by $\cR^{(\al)}_n$ its commutant,
i.e., the von Neumann algebra that consists
of all bounded linear operators
acting in $\cA_n^2(\bD,\mu_\al)$
that commute with $\rho^{(\al)}_n(\tau)$
for every $\tau$ in $\bT$.
In other words, the elements of $\cR^{(\al)}_n$ are the operators
intertwining the representation $\rho^{(\al)}_n$.
The elements of $\cR^{(\al)}_n$
are called \emph{radial} operators in $\cA_n^2(\bD,\mu_\al)$.

In a similar manner, we denote by $\rho^{(\al)}_{(n)}(\tau)$
the rotation operators acting in $\cA_{(n)}^2(\bD,\mu_\al)$
and by $\cR^{(\al)}_{(n)}$
the von Neumann algebra
of radial operators in $\cA_{(n)}^2(\bD,\mu_\al)$.
We also consider the rotation operators
$\rho^{(\al)}(\tau)$ in $L^2(\bD,\mu_\al)$
and the corresponding algebra $\cR^{(\al)}$
of radial operators.

\subsection*{Structure and results of this paper}

\begin{itemize}
\item In Section~\ref{sec:Jacobi_polynomials},
we list some necessary facts about Jacobi polynomials.
They play a crucial role in Sections~\ref{sec:basis}
and \ref{sec:weighted_mean_value}.

\item In Section~\ref{sec:basis},
we recall various equivalent formulas for the disk polynomials that can be obtained by orthogonalizing the monomials in $z$ and $\conjz$.
Using this orthonormal basis $(\basic{\al}{p}{q})_{p,q\in\bNz}$
we descompose $L^2(\bD,\mu_\al)$ into the orthogonal sum of subspaces $\FreqSubspace_\xi^{(\al)}$ corresponding to different frequences $\xi$, $\xi\in\bZ$.

\item In Section~\ref{sec:weighted_mean_value} we give an elementary proof of the weighted mean value property of polyanalytic functions and show the boundedness of the evaluation functionals for the spaces of polyanalytic functions over general domains in $\bC$.
In the unweighted case, this mean value property was proven
by Koshelev~\cite{Koshelev1977} and Pessoa~\cite{Pessoa2014}.
In the weighted case, it was found by
Hachadi and Youssfi~\cite{HachadiYoussfi2019}
and used by them to compute the reproducing kernel of $\cA_n^2(\bD,\mu_\al)$.

\item In Section~\ref{sec:spaces},
extending results by Ramazanov~\cite{Ramazanov1999,Ramazanov2002}
to the weighted case, we verify that the family
$(\basic{\al}{p}{q})_{p\ge 0,0\le q<n}$
is an orthonormal basis of $\cA_n^2(\bD,\mu_\al)$.
Using this fact, we decompose $\cA_n^2(\bD,\mu_\al)$
into subspaces
$\FreqSubspace_\xi^{(\al)}
\cap \cA_n^2(\bD,\mu_\al)$.

\item In Section~\ref{sec:not_weakly_dense} we prove that the set of all Toeplitz operators with bounded generating symbols
is not weakly dense in $\cB(\cA_n^2(\bD,\mu_\al))$.
This simple result was surprising for us.

\item In Section~\ref{sec:radial}
we decompose the von Neumann algebras
$\cR^{(\al)}$, $\cR^{(\al)}_n$, and $\cR^{(\al)}_{(n)}$,
into direct sums of factors.
In particular, Theorems~\ref{thm:radial_polyanalytic_Bergman}
and \ref{thm:radial_true_polyanalytic_Bergman}
imply that the algebra $\cR^{(\al)}_n$
is noncommutative for $n\ge2$,
whereas $\cR^{(\al)}_{(n)}$ is commutative
for every $n$ in $\bN$.

\item In Section~\ref{sec:Toeplitz},
we find explicit representations
of the radial Toeplitz operators
acting in the spaces $\cA_n^2(\bD,\mu_\al)$
and $\cA_{(n)}^2(\bD,\mu_\al)$.
The results of Sections~\ref{sec:radial}
and~\ref{sec:Toeplitz} are similar
to~\cite{MaximenkoTelleriaRomero2020}.
The main difference is that
the orthonormal bases are given by other formulas.
\end{itemize}
Most of the facts in the Sections~\ref{sec:Jacobi_polynomials}--\ref{sec:spaces} are not new. We recall them in a self-contained form, for reader's convenience.

We hope that this paper can serve as a basis for further investigations about polyanalytic or polyharmonic Bergman spaces and operators acting on these spaces.
For example, an interesting task is to describe the C*-algebra generated by radial Toeplitz operators acting in $\cA_n^2(\bD,\mu_\al)$.

\section{Necessary facts about Jacobi polynomials}
\label{sec:Jacobi_polynomials}

In this section, we recall some necessary facts about Jacobi polynomials.
Most of them are explained in
\cite[Chapter~4]{Szego1975}.
For every $\al$ and $\be$ in $\bR$,
the Jacobi polynomials can be defined by Rodrigues formula:
\begin{equation}\label{eq:Jacobi_Rodrigues}
P_n^{(\al,\be)}(x)
\eqdef
\frac{(-1)^n}{2^n\,n!}\,(1-x)^{-\al}(1+x)^{-\be}
\frac{\dif^n}{\dif{}x^n}
\Bigl((1-x)^{n+\al} (1+x)^{n+\be} \Bigr).
\end{equation}
This definition and the general Leibniz rule imply an expansion
into powers of $x-1$ and $x+1$:
\begin{equation}\label{eq:Jacobi_explicit}
P_n^{(\al,\be)}(x)
=\sum_{k=0}^n \binom{n+\al}{n-k}\,\binom{n+\be}{k}
\left(\frac{x-1}{2}\right)^k \left(\frac{x+1}{2}\right)^{n-k}.
\end{equation}
Formula~\eqref{eq:Jacobi_explicit}
yields a symmetry relation,
the values at the points $1$ and $-1$,
and a formula for the derivative:
\begin{equation}\label{eq:Jacobi_symmetry}
P_n^{(\al,\be)}(-x)=(-1)^n P_n^{(\be,\al)}(x),
\end{equation}
\begin{equation}\label{eq:Jacobi_extreme_values}
P_n^{(\al,\be)}(1)=\binom{n+\al}{n},\qquad
P_n^{(\al,\be)}(-1)=(-1)^n\binom{n+\be}{n},
\end{equation}
\begin{equation}\label{eq:Jacobi_derivative}
\bigl(P_n^{(\al,\be)}\bigr)'(x)
=\frac{\al+\be+n+1}{2}\,P_{n-1}^{(\al+1,\be+1)}(x).
\end{equation}
With the above properties, it is easy to compute
the derivatives of $P_n^{(\al,\be)}$ at the point $1$.
Now Taylor's formula yields another explicit expansion for $P_n^{(\al,\be)}$:
\begin{equation}\label{eq:Jacobi_explicit2}
P_n^{(\al,\be)}(x)
=\sum_{k=0}^n \binom{\al+\be+n+k}{k}\binom{\al+n}{n-k}
\left(\frac{x-1}{2}\right)^k.
\end{equation}
If $\al>-1$ and $\be>-1$,
then \eqref{eq:Jacobi_explicit2} can be rewritten as
\begin{equation}\label{eq:Jacobi_explicit3}
P_n^{(\al,\be)}(x)
=\frac{\Ga(\al+n+1)}{n!\,\Gamma (\al+\be+n+1)}
\sum_{k=0}^n \binom{n}{k} \frac{\Ga(\al+\be+n+k+1)}{\Ga(\al+k+1)}
\left(\frac{x-1}{2}\right)^k.
\end{equation}
For $\al>-1$ and $\be>-1$,
we equip $(-1,1)$ with the weight $(1-x)^\al(1+x)^\be$,
then denote by $\langle\cdot,\cdot\rangle_{(-1,1),\al,\be}$
the corresponding inner product:
\[
\langle f,g\rangle_{(-1,1),\al,\be}
\eqdef\int_{-1}^1 f(x)\overline{g(x)}\,(1-x)^\al(1+x)^\be\,\dif{}x.
\]
Then $L^2((-1,1),(1-x)^\al(1+x)^\be)$ is a Hilbert space,
and the set $\cP$ of the univariate polynomials
is its dense subset.
Using~\eqref{eq:Jacobi_Rodrigues}
and integrating by parts, for every $f$ in $\cP$ we get
\begin{equation}\label{eq:inner_product_with_Jacobi}
\langle f,P_n^{(\al,\be)}\rangle_{(-1,1),\al,\be}
=\frac{1}{2n}
\langle f',P_{n-1}^{(\al+1,\be+1)}\rangle_{(-1,1),\al+1,\be+1}.
\end{equation}
Applying~\eqref{eq:inner_product_with_Jacobi} and induction,
it is easy to prove that the sequence
$(P_n^{(\al,\be)})_{n=0}^\infty$
is an orthogonal basis of $L^2((-1,1),(1-x)^\al(1+x)^\be)$,
that is, for every polynomial $h$ of degree less than $n$,
\begin{equation}\label{eq:Jacobi_ortogonality}
    \int_{-1}^1 h(x) P_n^{(\al,\be)}(x)(1-x)^\al (1+x)^\be\dif{}x = 0.
\end{equation}
Furthermore,
\begin{equation}\label{eq:Jacobi_norm}
\langle
P_m^{(\al,\be)},P_n^{(\al,\be)}
\rangle_{(-1,1),\al,\be}
=\frac{2^{\al+\be+1}\,\Ga(n+\al+1)\Ga(n+\be+1)}{(2n+\al+\be+1)\Ga(n+\al+\be+1)\,n!}\delta_{m,n}.
\end{equation}
Formulas~\eqref{eq:Jacobi_Rodrigues} and~\eqref{eq:Jacobi_ortogonality}, and induction
allow us to compute the following integral for $\be>0$:
\begin{equation}\label{eq:Jacobi_with_incomplete_weight}
\int_{-1}^1 P_n^{(\al,\be+1)}(x)\,(1-x)^\al(1+x)^\be\,\dif{}x
=2^{\al+\be+1}(-1)^n \Be(\al+n+1,\be+1),
\end{equation}
where $\Be$ is the well-known Beta function. 

\subsection*{Jacobi polynomials for the unit interval}

The function $t\mapsto 2t-1$
is a bijection from $(0,1)$ onto $(-1,1)$.
Denote by $Q_n^{(\al,\be)}$
the ``shifted Jacobi polynomial''
obtained from $P_n^{(\al,\be)}$
by composing it with this change of variables:
\[
Q_n^{(\al,\be)}(t)\eqdef P_n^{(\al,\be)}(2t-1).
\]
The properties of $Q_n^{(\al,\be)}$
follow easily from the properties of $P_n^{(\al,\be)}$.
In particular, here are analogs
of~\eqref{eq:Jacobi_Rodrigues},
\eqref{eq:Jacobi_explicit2},
and \eqref{eq:Jacobi_explicit3}:
\begin{align}
\label{eq:shifted_Jacobi_Rodrigues}
Q_n^{(\al,\be)}(t)
&=\frac{(-1)^n}{n!}\,(1-t)^{-\al} t^{-\be}
\frac{\dif^n}{\dif{}t^n}
\Bigl((1-t)^{n+\al} t^{n+\be}\Bigr),
\\
\label{eq:shifted_Jacobi_explicit2}
Q_n^{(\al,\be)}(t)
&=\sum_{k=0}^n \binom{\al+\be+n+k}{k}\binom{\be+n}{n-k}
(-1)^{n-k}t^k.
\\
\label{eq:shifted_Jacobi_explicit3}
Q_n^{(\al,\be)}(t)
&=\frac{\Ga(n+\be+1)}{n!\,\Ga(n+\al+\be+1)}
\sum_{k=0}^n \binom{n}{k}
\frac{\Ga(\al+\be+n+k+1)}{\Ga(\be+k+1)}\,
(-1)^{n-k}t^k.
\end{align}
The sequence $(Q_n^{(\al,\be)})_{n=0}^\infty$
is orthogonal on $(0,1)$ with respect to the weight $(1-t)^\al t^\be$,
and
\begin{equation}\label{eq:Q_inner_prod}
    \int_0^1 Q_m^{(\al,\be)}(t)
    Q_n^{(\al,\be)}(t)(1-t)^\al t^\be \dif{}t =\de_{m,n}
    \frac{\Ga(n+\al+1)\Ga(n+\be+1)}{(2n+\al+\be+1)\Ga(n+\al+\be+1)\,n!}.
\end{equation}
Also, here are analogs of~\eqref{eq:Jacobi_ortogonality}
and~\eqref{eq:Jacobi_with_incomplete_weight}:
\begin{equation}\label{eq:shifted_Jacobi_orothogonality}
\int_0^1 h(t) Q_n^{(\al,\be)}(t)(1-t)^\al t^\be\dif{}t = 0,
\end{equation}
\begin{equation}\label{eq:shifted_Jacobi_with_incomplete_weight}
\int_0^1 Q_n^{(\al,\be+1)}(t)(1-t)^\al t^\be \dif{}t =(-1)^n\Be(\al+n+1,\be+1).
\end{equation}
Substituting in~\eqref{eq:shifted_Jacobi_Rodrigues}
$t$ by $tu$ and applying the chain rule, we get
\begin{equation}
\label{eq:shifted_Jacobi_Rodrigues_for_tu}
\frac{\partial^n}{\partial{}t^n}
\Bigl((1-tu)^{n+\al} t^{n+\be}\Bigr)
=n!\,(1-tu)^\al\,t^\be\,Q^{(\al,\be)}_n(tu).
\end{equation}
Inspired by~\eqref{eq:Q_inner_prod}
we define the function $\jac_n^{(\al,\be)}$ on $(0,1)$ as
\begin{equation}\label{eq:jac}
\jac_n^{(\al,\be)}(t)
\eqdef \jaccoef{\al}{\be}{n}
(1-t)^{\al/2}t^{\be/2}Q_n^{(\al,\be)}(t),
\end{equation}
where
\begin{equation}\label{eq:jaccoef}
\jaccoef{\al}{\be}{n}
\eqdef
\sqrt{\frac{(2n+\al+\be+1)\,\Ga(n+\al+\be+1)\,n!}
{\Ga(n+\al+1)\Ga(n+\be+1)}}.
\end{equation}
Then
\begin{equation}\label{eq:jac_orthonormality}
\int_0^1 \jac_m^{(\al,\be)}(t)\jac_n^{(\al,\be)}(t)
\,\dif{}t
=\de_{m,n}.
\end{equation}

\subsection*{Reproducing property for the polynomials on the unit interval}

Given $n$ in $\bNz$ and $\al,\be>-1$,
we denote by $R_n^{(\al,\be)}$ the polynomial
\begin{equation}\label{eq:Rpol_def}
R_n^{(\al,\be)}(t)
\eqdef \frac{(-1)^n\,\Be(\al+1,\be+1)}{\Be(\al+n+1,\be+1)}\,Q_n^{(\al,\be+1)}(t).
\end{equation}

\begin{proposition}\label{prop:R_repr_on_interval}
Let $n\in\bNz$ and $\al,\be>-1$.
Then for every polynomial $h$ with $\deg(h)\le n$,
\begin{equation}\label{eq:R_repr_on_interval}
\frac{1}{\Be(\al+1,\be+1)}\int_0^1 h(t) R_n^{(\al,\be)}(t)\,(1-t)^\al t^\be\,\dif{}t=h(0).
\end{equation}
\end{proposition}

\begin{proof}
The difference $h(t)-h(0)$ divides by $t$.
Denote by $q(t)$ the corresponding quotient.
So, $q$ is a polynomial of degree $\deg(q)\le n-1$
and $h(t)=h(0)+t q(t)$.
Then
\begin{align*}
\frac{1}{\Be(\al+1,\be+1)}&\int_0^1 h(t) R_n^{(\al,\be)}(t)\,(1-t)^\al t^\be\,\dif{}t\\
&=
\frac{(-1)^n h(0)}{\Be(\al+n+1,\be+1)} 
\int_0^1 Q_n^{(\al,\be+1)}(t)\,(1-t)^\al t^\be\,\dif{}t
\\
&\quad + \frac{(-1)^n}{\Be(\al+n+1,\be+1)}
\int_0^1 q(t) Q_n^{(\al,\be+1)}(t)\,(1-t)^\al t^{\be+1}\,\dif{}t.
\end{align*}
By~\eqref{eq:shifted_Jacobi_with_incomplete_weight}
and~\eqref{eq:shifted_Jacobi_orothogonality},
the first summand is $h(0)$ and the second is $0$.
\end{proof}

As a particular case of \eqref{eq:R_repr_on_interval}, for $\be=0$ and $m\le n$, 
\begin{equation}\label{eq:R_interval_k_delta}
    \frac{1}{\al+1}\int_0^1 t^m R_n^{(\al,0)}(t)(1-t)^\al\dif{}t=\delta_{m,0}.
\end{equation}
Formula~\eqref{eq:R_interval_k_delta} was proven in~\cite{HachadiYoussfi2019} in other way.

\section{\texorpdfstring{Orthonormal basis and Fourier decomposition of $\boldsymbol{L^2(\bD,\mu_\al)}$}{Orthonormal basis and Fourier decomposition of L2(D,mualpha)}}
\label{sec:basis}

For each $p, q \in \bNz$, denote by $m_{p,q}$ the monomial function 
\[
m_{p,q}(z)\eqdef z^{p}\conjz^{q}.
\]
The inner product of two monomial functions is
\begin{equation}\label{eq:monomial_inner_product}
\langle m_{p,q},m_{j,k}\rangle
=(\al+1)\Be\left(p+j+1,\al+1\right)\,\de_{p-q,j-k}.
\end{equation}
In particular, this means that the family $(m_{p,q})_{p,q\in \mathbb{N}_0}$ is not orthogonal.

In this section, we recall various equivalent formulas for an orthonormal basis in $L^2(\bD,\mu_\al)$,
that can be obtained by orthonormalizing $(m_{p,q})_{p,q\in\bNz}$,
and whose elements are known as \emph{Jacobi polynomials in $z$ and $\conjz$}, see Koornwinder~\cite{Koornwinder1975},
or \emph{disk polynomials},
see~W\"{u}nsche~\cite{Wunsche2005}, among others.
These polynomials, in the unweighted case,
were also rediscovered in~\cite{Koshelev1977}, \cite{Ramazanov1999}, and \cite{Pessoa2014},
in the context of polyanalytic functions.
We work with a normalized version of the disk polynomials and define them by
\begin{equation}\label{eq:b_pq_by_Ramazanov}
\basic{\al}{p}{q}(z)
\eqdef (-1)^{p+q}\, \bcoef{\al}{p}{q}\,(1-z\conjz)^{-\alpha}\;\frac{\partial^q}{\partial z^q}\,\frac{\partial^p}{\partial{}\conjz^p}
\Bigl((1-z\,\conjz)^{p+q+\alpha}\Bigr),
\end{equation}
where
\begin{equation}\label{eq:c_alpha_p_q}
\bcoef{\al}{p}{q}
=\sqrt{\frac{(\alpha+p+q+1)\Ga(\alpha+p+1)\Ga(\alpha+q+1)}%
{(\alpha+1)p!\,q!\,\Ga(\al+p+q+1)^2}}.
\end{equation}
Since $\displaystyle\frac{\partial}{\partial z}(1-z\conjz)=-\conjz$ and $\displaystyle\frac{\partial}{\partial\conjz}(1-z\conjz)=-z$, the expression in \eqref{eq:b_pq_by_Ramazanov} can be rewritten in other equivalent forms:
\begin{align}
\basic{\al}{p}{q}(z)&=(-1)^{q}\sqrt{\frac{(\alpha+p+q+1)\Ga(\alpha+p+1)}%
{(\alpha+1)p!\,q!\,\Ga(\al+q+1)}}(1-z\overline{z})^{-\alpha}\,\frac{\partial^q}{\partial z^q}
\Bigl(z^p(1-z\,\conjz)^{\al+q}\Bigr),\\
\basic{\al}{p}{q}(z)&=(-1)^{p}\sqrt{\frac{(\alpha+p+q+1)\Ga(\alpha+q+1)}%
{(\alpha+1)p!\,q!\,\Ga(\al+p+1)}}(1-z\overline{z})^{-\alpha}\,\frac{\partial^p}{\partial\overline{z}^p}
\Bigl(\overline{z}^q(1-z\,\conjz)^{\al+p}\Bigr).
\end{align}
By~\eqref{eq:shifted_Jacobi_Rodrigues_for_tu}, $\basic{\al}{p}{q}$ can be expressed via the shifted Jacobi polynomials:
\begin{equation}\label{eq:b_Jacobi}
\basic{\al}{p}{q}(z)
=
\begin{cases}
\sqrt{\dfrac{(\alpha+p+q+1)\Ga(\alpha+p+1)q!}%
{(\alpha+1)\,\Ga(\al+q+1)p!}}\,z^{p-q}Q_{q}^{(\al, p-q)}(|z|^2), & \text{if}\ p\ge q;\\[2ex]
\sqrt{\dfrac{(\alpha+p+q+1)\Ga(\alpha+q+1)p!}%
{(\alpha+1)\,\Ga(\al+p+1)q!}}\,\overline{z}^{q-p}Q_{p}^{(\al, q-p)}(|z|^2), & \text{if}\ p< q.
\end{cases}
\end{equation}
The two cases in \eqref{eq:b_Jacobi} can be joined
and written in terms of~\eqref{eq:jac} and~\eqref{eq:jaccoef}:
\begin{align}\label{eq:b_via_Q}
\basic{\al}{p}{q}(r\tau)
&=\frac{\jaccoef{\al}{|p-q|}{\min\{p,q\}}}{\sqrt{\al+1}}
r^{|p-q|}\tau^{p-q}
Q_{\min\{p,q\}}^{(\al,|p-q|)}(r^2)
\qquad(r\ge0,\ \tau\in\bT),
\\
\label{eq:b_via_jac}
\basic{\al}{p}{q}(r\tau)
&=\frac{\tau^{p-q}(1-r^2)^{-\al/2}}{\sqrt{\al+1}}
\jac_{\min\{p,q\}}^{(\al,|p-q|)}(r^2).
\end{align}
Notice that
\[
\jaccoef{\al}{|p-q|}{\min\{p,q\}}
=\sqrt{\frac{(\al+p+q+1)(\min\{p,q\})!\,\Ga(\al+\max\{p,q\}+1)}%
{(\max\{p,q\})!\,\Ga(\al+\min\{p,q\}+1)}}.
\]
The family $(\basic{\al}{p}{q})_{p,q\in\bNz}$
has the following conjugate symmetric property:
\begin{equation}\label{eq:conjugate_symmetric_property_of_basics}
\overline{\basic{\al}{p}{q}(z)}=\basic{\al}{q}{p}(z).
\end{equation}
Applying~\eqref{eq:shifted_Jacobi_explicit3} in the right-hand side of \eqref{eq:b_Jacobi} we obtain
\begin{equation}\label{eq:b_pq_in_shifted_Jacobi_explicit3_form}
\begin{aligned}
    \basic{\al}{p}{q}(z)&=\sqrt{\frac{(\alpha+p+q+1)p!\,q!}%
{(\alpha+1)\,\Ga(\al+p+1)\Ga(\al+q+1)}}\times\\
&\qquad\times\sum_{k=0}^{\min\{p,q\}}(-1)^{k}\frac{\Ga(\al+p+q+1-k)}{k!\,(p-k)!\,(q-k)!}z^{p-k}\conjz^{q-k},
\end{aligned}
\end{equation}
In particular,~\eqref{eq:b_pq_in_shifted_Jacobi_explicit3_form} implies that $\basic{\al}{p}{q}$ is a polynomial in $z$ and $\conjz$ whose leading term is a positive multiple of the monomials $m_{p-k,q-k}$.

Let $\mathcal{P}$ be the set of all polynomials functions in $z$ and $\conjz$, i.e., the linear span of the monomials
\[
\mathcal{P}\eqdef \lin\{m_{p,q}\colon p,q\in\bNz\}.
\]
For every $\xi \in \mathbb{Z}$ and every $s\in \mathbb{N}$, denote by $\FreqSubspace^{(\al)}_{\xi,s}$ the subspace of $\mathcal{P}$ generated by $m_{p,q}$ with $p-q=\xi$ and $\min\{p,q\}< s$:
\begin{equation}\label{def:fiber_subspace}
\FreqSubspace^{(\al)}_{\xi,s}\eqdef\lin\{m_{p,q}\colon\  p-q=\xi,\ \min\{p,q\}< s\}.
\end{equation}
The vector space $\FreqSubspace^{(\al)}_{\xi,s}$ does not depend on $\al$, but we endow it with the inner product from $L^2(\bD,\mu_\al)$.
Obviously, $\dim(\FreqSubspace^{(\al)}_{\xi,s})=s$.
Let us show that
\begin{equation}\label{eq:truncated_FreqSbpce_gen_by_b}
\FreqSubspace^{(\al)}_{\xi,s}
=\lin\{\basic{\al}{p}{q}\colon\ p-q=\xi,\ \min\{p,q\}< s\}.
\end{equation}
Indeed, by~\eqref{eq:b_pq_in_shifted_Jacobi_explicit3_form},
\begin{equation}\label{eq:m_in_terms_of_b_and_younger_brothers}
\begin{aligned}
 m_{p,q}&=\sqrt{\frac{(\al+1)\Ga(\al+p+1)\Ga(\al+q+1)p!\,q!}
{\Ga(\alpha+p+q+2)}}\,\basic{\al}{p}{q}\\
&\quad-\frac{p!\,q!}{\Ga(\alpha+p+q+1)}\,\sum_{\nu=1}^{\min\{p,q\}}(-1)^{\nu}\frac{\Ga(\al+p+q+1-\nu)}{\nu!\,(p-\nu)!\,(q-\nu)!}\,m_{p-\nu,q-\nu}.
\end{aligned}
\end{equation}
Proceeding by induction on $s$, we see that the monomials $ m_{p,q}$ are linear combinations of $\basic{\al}{p-s}{q-s}$ with $ 0\leq s\leq \min \{ p,q\}$.
So, formula~\eqref{eq:truncated_FreqSbpce_gen_by_b} means that the first $s$ elements in the diagonal $\xi$ of the table $(\basic{\al}{p}{q})_{p,q=0}^\infty$ generate the same subspace as the first $s$ elements of the diagonal $\xi$ in the table $(m_{p,q})_{p,q=0}^\infty$.
For example,
\begin{align*}
\FreqSubspace^{(\al)}_{-2,3}
&=\lin\{m_{0,2},m_{1,3},m_{2,4}\}=\lin\{\basic{\al}{0}{2},\basic{\al}{1}{3},\basic{\al}{2}{4}\},\\
\FreqSubspace^{(\al)}_{1,4}
&=\lin\{m_{1,0},m_{2,1},m_{3,2},m_{4,3}\}=\lin\{\basic{\al}{1}{0},\basic{\al}{2}{1},\basic{\al}{3}{2},\basic{\al}{4}{3}\}.
\end{align*}
In the following tables we show generators of $\FreqSubspace^{(\al)}_{1,4}$ (light blue) and $\FreqSubspace^{(\al)}_{-2,3}$ (pink).
\[
\begin{array}{ccccccc}
m_{0,0} & m_{0,1} & \cellcolor{pink!60!}m_{0,2} & m_{0,3} & m_{0,4} & m_{0,5} & \ddots \\
\cellcolor{lightblue}m_{1,0} & m_{1,1} & m_{1,2} & \cellcolor{pink!60!}m_{1,3} & m_{1,4} & m_{1,5} & \ddots \\
m_{2,0} &\cellcolor{lightblue} m_{2,1} & m_{2,2} & m_{2,3} & \cellcolor{pink!60!}m_{2,4} & m_{2,5} & \ddots \\
m_{3,0} & m_{3,1} & \cellcolor{lightblue}m_{3,2} & m_{3,3} & m_{3,4} & m_{3,5} & \ddots \\
m_{4,0} & m_{4,1} & m_{4,2} & \cellcolor{lightblue}m_{4,3} & m_{4,4} & m_{4,5} & \ddots \\
m_{5,0} & m_{5,1} & m_{5,2} & m_{5,3} & m_{5,4} & m_{5,5} & \ddots \\
\ddots & \ddots & \ddots & \ddots & \ddots & \ddots & \ddots
\end{array}
\qquad\quad
\begin{array}{ccccccc}
\basic{\al}{0}{0} & \basic{\al}{0}{1} & \cellcolor{pink!60!}\basic{\al}{0}{2} & \basic{\al}{0}{3} & \basic{\al}{0}{4} & \basic{\al}{0}{5} & \ddots \\
\cellcolor{lightblue}\basic{\al}{1}{0} & \basic{\al}{1}{1} & \basic{\al}{1}{2} & \cellcolor{pink!60!}\basic{\al}{1}{3} & \basic{\al}{1}{4} & \basic{\al}{1}{5} & \ddots \\
\basic{\al}{2}{0} & \cellcolor{lightblue}\basic{\al}{2}{1} & \basic{\al}{2}{2} & \basic{\al}{2}{3} & \cellcolor{pink!60!}\basic{\al}{2}{4} & \basic{\al}{2}{5} & \ddots \\
\basic{\al}{3}{0} & \basic{\al}{3}{1} & \cellcolor{lightblue}\basic{\al}{3}{2} & \basic{\al}{3}{3} & \basic{\al}{3}{4} & \basic{\al}{3}{5} & \ddots \\
\basic{\al}{4}{0} & \basic{\al}{4}{1} & \basic{\al}{4}{2} &\cellcolor{lightblue}\basic{\al}{4}{3} & \basic{\al}{4}{4} & \basic{\al}{4}{5} & \ddots \\
\basic{\al}{5}{0} & \basic{\al}{5}{1} & \basic{\al}{5}{2} & \basic{\al}{5}{3} & \basic{\al}{5}{4} & \basic{\al}{5}{5} & \ddots \\
\ddots & \ddots & \ddots & \ddots & \ddots & \ddots & \ddots
\end{array}
\]
As a consequence, $\mathcal{P}=\bigcup_{\xi\in \mathbb{Z}}\bigcup_{s\in \mathbb{N}}\FreqSubspace^{(\al)}_{\xi,s}=\lin\{\basic{\al}{p}{q}\colon p,q\in\bNz\}$. 

\begin{proposition}\label{prop:b_is_orthonormal_basis_in_L2}
The family $(\basic{\al}{p}{q})_{p,q\in\bNz}$ is
an orthonormal basis of $L^{2}(\bD, \mu_{\al})$.
\end{proposition}
\begin{proof}
The orthonormal property follows straightforwardly from~\eqref{eq:b_via_jac} and~\eqref{eq:jac_orthonormality}:
\begin{align*}
    \langle \basic{\al}{p}{q},\basic{\al}{j}{k}\rangle
&=\frac{1}{2\pi}\int_{0}^{2\pi}
\econstant^{\imagunit(p-q-j+k)\theta}\,\dif{\theta}\int_{0}^{1}\jac_{\min\{p,q\}}^{(\al,|p-q|)}(t)\,
\jac_{\min\{j,k\}}^{(\al,|j-k|)}(t)\,\dif{}t\\
&= \de_{p-q,j-k}\cdot\de_{\min\{p,q\},\min\{j,k\}}
=\de_{p,j}\cdot\de_{q,k}.
\end{align*}
By the Stone--Weierstrass theorem, $\mathcal{P}$ is dense in $C(\clos(\bD))$. In turn, by Luzin's theorem, the set $C(\clos(\bD))|_{\bD}$ is dense in $L^2(\bD, \mu_{\al})$ and for every $f \in C(\clos(\bD))$ we have $\|f\|\leq \max_{z\in \clos(\bD)}|f(z)|$. Now it is easy to see that the set $\mathcal{P}$ is dense in $L^2(\bD, \mu_{\al})$, that is, the set of all linear combinations of elements of the family is dense in $L^2(\bD, \mu_{\al})$. For that reason, $(\basic{\al}{p}{q})_{p,q\in\bNz}$
is a complete orthonormal family.
\end{proof}

\begin{corollary}\label{cor:basis_of_truncated_FreqSubspace}
Let $\xi\in\bZ$ and $s\in\bN$.
Then $(\basic{\al}{q+\xi}{q})_{q=\max\{0,-\xi\}}^{\max\{s-1,s-\xi-1\}}$
is an orthonormal basis of~
$\FreqSubspace^{(\al)}_{\xi,s}$.
\end{corollary}

\begin{remark}\label{rem:products_b_m}
By Proposition~\ref{prop:b_is_orthonormal_basis_in_L2} and formula~\eqref{eq:m_in_terms_of_b_and_younger_brothers}, 
\begin{equation}\label{eq:product_b_m}
\langle m_{\xi+k,k},\basic{\al}{\xi+q}{q}\rangle
=
\begin{cases}
\sqrt{\frac{(\al+1)\Ga(\al+p+1)\Ga(\al+q+1)p!\,q!}
{\Ga(\alpha+p+q+2)\Ga(\alpha+p+q+1)}},\, & k=q; \\[0.5ex]
0, & \max\{0,-\xi \}\leq k<q.
\end{cases}
\end{equation}
\end{remark}

The table of basic functions can be expressed as follows:
\begin{align*}
\basic{\al}{0}{0}(z) &= h^{(\al)}_{0,0}(|z|^2), &
\basic{\al}{0}{1}(z) &=  \conjz\,h^{(\al)}_{0,1}(|z|^2), &
\basic{\al}{0}{2}(z)&= \conjz^{2}\,h^{(\al)}_{0,2}(|z|^2),
\\[1ex]
\basic{\al}{1}{0}(z) &= z\,h^{(\al)}_{1,0}(|z|^2), &
\basic{\al}{1}{1}(z)&= h^{(\al)}_{1,1}(|z|^2), &
\basic{\al}{1}{2}(z)&=\conjz\,h^{(\al)}_{1,2}(|z|^2),
\\[1ex]
\basic{\al}{2}{0}(z) &= z^{2}\,h^{(\al)}_{2,0}(|z|^2), &
\basic{\al}{2}{1}(z)&=z\,h^{(\al)}_{2,1}(|z|^2), &
\basic{\al}{2}{2}(z)&=h^{(\al)}_{2,2}(|z|^2),
\end{align*}
where $h^{(\al)}_{p,q}(t)\eqdef\frac{\jaccoef{\al}{|p-q|}{\min\{p,q\}}}{\sqrt{\al+1}}
Q_{\min\{p,q\}}^{(\al,|p-q|)}(t)$.
Below we show explicitly some elements of this basis.
\begin{align*}
\basic{\al}{0}{0}(z) &= 1, \qquad 
\basic{\al}{1}{0}(z)= \sqrt{\al +2}\,z, \qquad 
\basic{\al}{2}{0}(z)= \sqrt{\frac{(\al+3)(\al+2)}{2}}\,z^2,
\\[1ex]
\basic{\al}{1}{1}(z)
&= \sqrt{\displaystyle\frac{\al +3}{\al +1}}\big ((\al +2)z\conjz-1 \big ), \qquad
\basic{\al}{2}{1}(z)
=\sqrt{\displaystyle\frac{2(\al +3)(\al +2)}{\al +1}}
\left(\frac{\al+3}{2}z^2\conjz
-z \right),
\\[1ex]
\basic{\al}{2}{2}(z)
&=\sqrt{\displaystyle\frac{\al +5}{\al +1}}
 \left(
 \frac{(\al +4)(\al+3)}{2} z^2\conjz^2
 -2(\al +3)z\conjz
 +\frac{1}{2} \right).
\end{align*}
Now, for every $\xi$ in $\bZ$ we introduce the subspace $\FreqSubspace_\xi^{(\al)}$
associated to the ``frequency'' $\xi$
or, equivalently, to the diagonal $\xi$ in the tables $(m_{p,q})_{p,q\in\bNz}$
and $(\basic{\al}{p}{q})_{p,q\in\bZ}$:
\[
\FreqSubspace_\xi^{(\al)}
\eqdef\clos(\lin\{m_{p,q}\colon p-q=\xi\})
=\clos\Biggl(\,\displaystyle
\bigcup_{s\in \bN}
\FreqSubspace_{\xi,s}^{(\al)}\,
\Biggr).
\]

\begin{corollary}\label{cor:basis_of_FreqSubspace}
The sequence $(\basic{\al}{\xi+q}{q})_{q=\max\{0,-\xi\}}^\infty$
is an orthonormal basis of $\FreqSubspace_\xi^{(\al)}$.
\end{corollary}

The space $\FreqSubspace_\xi^{(\al)}$ can be naturally identified with $L^2$ over $(0,1)$, providing $(0,1)$ with various weights.

\begin{proposition}
\label{prop:polar_description_of_W}
Each one of the following linear operators is an isometric isomorphism of Hilbert spaces:
\begin{itemize}
\item[1)]
$L^2((0,1),(\al+1)\,(1-t)^{\al}\dif t)
\to \FreqSubspace_\xi^{(\al)}$,\quad
$h\mapsto f$,\quad
\begin{equation}\label{polar_description_of_D1}
f(r\tau)\eqdef\tau^{\xi} h(r^2),\quad\text{i.e.,} \quad
f(z)\eqdef\sign^{\xi} (z)h(z\conjz),
\end{equation}
where $z\in\bD$, $0\le r<1$, $\tau\in\bT$;
\item[2)]
$L^2((0,1),(\al+1)\,t^{|\xi|}(1-t)^{\al})
\to \FreqSubspace_\xi^{(\al)},$\quad
$h\mapsto f$,\quad
\begin{equation}\label{polar_description_of_D2}
f(r\tau)\eqdef \tau^{\xi}r^{|\xi|} h(r^2),\quad\text{i.e.,} \quad 
f(z)\eqdef\begin{cases}
z^{\xi}h(z\conjz),& \xi\geq 0,\\
\conjz^{\xi}h(z\conjz), &\xi<0;
\end{cases}
\end{equation}
\item[3)]
$L^2((0,1))
\to \FreqSubspace_\xi^{(\al)}$,\quad
$h\mapsto f$,\quad
\begin{equation}\label{polar_description_of_D3}
f(r\tau)\eqdef \tau^{\xi}\frac{(1-r^2)^{-\al/2}}{\sqrt{\al+1}}\,h(r^2),\quad\text{i.e.,} \quad
f(z)\eqdef\sign^{\xi}(z)\frac{(1-z\conjz)^{-\al/2}}{\sqrt{\al+1}}h(z\conjz).
\end{equation}
\end{itemize}
\end{proposition}

\begin{proof}
In each case, the isometric property is verified directly using polar coordinates,
and the surjective property is justified
with the help of the orthonormal basis of $\FreqSubspace_\xi^{(\al)}$
(Corollary~\ref{cor:basis_of_FreqSubspace}).
The function $\sign\colon\bC\to\bC$ is defined by $\sign(z)\eqdef z/|z|$ for $z\ne 0$ and $\sign(0)\eqdef 0$.
\end{proof}

\begin{corollary}\label{cor:decomposition_L2_into_diagonals}
The space $L^2(\bD,\dif{}\mu_{\al})$
is the orthogonal sum of the subspaces $\FreqSubspace_{\xi}^{(\al)}$:
\begin{equation}\label{eq:decomposition_L2_into_diagonal_subspaces}
L^2(\bD,\dif{}\mu_{\al})
=\bigoplus_{\xi\in\bZ}\FreqSubspace_{\xi}^{(\al)}.
\end{equation}
\end{corollary}

The result of Corollary~\ref{cor:decomposition_L2_into_diagonals} can be seen as the \emph{Fourier decomposition} of the space $L^2(\bD,\dif{}\mu_{\al})$,
and each space $\FreqSubspace_\xi^{(\al)}$
corresponds to the \emph{frequency} $\xi$.

Here we show the generators of $\mathcal{W}_{0}^{(\al)}$ (pink) and $\mathcal{W}_{-1}^{(\al)}$ (light blue): 
\[
\begin{array}{ccccc}
\cellcolor{pink!60!}m_{0,0} &\cellcolor{lightblue} m_{0,1} & m_{0,2} & m_{0,3} & \ddots \\
m_{1,0} & \cellcolor{pink!60!}m_{1,1} &\cellcolor{lightblue} m_{1,2} & m_{1,3} & \ddots \\
m_{2,0} & m_{2,1} &\cellcolor{pink!60!} m_{2,2} &\cellcolor{lightblue} m_{2,3} & \ddots \\
m_{3,0} & m_{3,1} & m_{3,2} &\cellcolor{pink!60!} m_{3,3} & \cellcolor{lightblue}\ddots \\
\ddots & \ddots & \ddots & \ddots &\cellcolor{pink!60!} \ddots
\end{array}
\qquad\quad
\begin{array}{ccccc}
\cellcolor{pink!60!}\basic{\al}{0}{0} & \cellcolor{lightblue}\basic{\al}{0}{1} & \basic{\al}{0}{2} & \basic{\al}{0}{3} & \ddots \\
\basic{\al}{1}{0} & \cellcolor{pink!60!}\basic{\al}{1}{1} &\cellcolor{lightblue} \basic{\al}{1}{2} & \basic{\al}{1}{3} & \ddots \\
\basic{\al}{2}{0} & \basic{\al}{2}{1} & \cellcolor{pink!60!}\basic{\al}{2}{2} &\cellcolor{lightblue} \basic{\al}{2}{3} & \ddots \\
\basic{\al}{3}{0} & \basic{\al}{3}{1} & \basic{\al}{3}{2} & \cellcolor{pink!60!}\basic{\al}{3}{3} &\cellcolor{lightblue} \ddots \\
\ddots & \ddots & \ddots & \ddots & \cellcolor{pink!60!}\ddots
\end{array}
\]

\section{Weighted mean value property of polyanalytic functions}
\label{sec:weighted_mean_value}

It is known \cite[Section~1.1]{Balk1991}
that any $n$-analytic function can be expressed as a ``polynomial'' of degree $n-1$ in the variable $\conjz$ with $1$-analytic coefficients, that is, for any $f\in\cA_n(\bD)$, there exist analytic functions $g_0,\:g_1,\ldots,g_{n-1}$ in $\bD$ such that
\[
f(z) = \sum\limits_{k=0}^{n-1} g_k(z)\conjz^k
\qquad (z\in\bD).
\]
Replacing every $g_k$ by its Taylor series,
we get another classic form of $n$-analytic functions: there exist coeficients $\la_{j,k}$ in $\bC$ such that
\begin{equation}\label{aux:poly_expansion}
f(z) = \sum\limits_{k=0}^{n-1}\sum\limits_{j=0}^{\infty} \la_{j,k}z^j\conjz^k
\qquad (z\in\bD).
\end{equation}

The following weighted mean value property
was proved by Hachadi y Youssfi~\cite{HachadiYoussfi2019}
using a slightly different method.
The mean value property for solutions of more general elliptic equations is studied in~\cite{Trofymenko2018}.

\begin{proposition}\label{prop:weighted_mean_value_property}
Let $f\in\cA_n(\bD)$ such that
\[
\int_{\bD}|f(z)|\,(1-|z|^2)^\al\,\dif\mu(z)<+\infty.
\]
Then
\begin{equation}\label{eq:weighted_mean_value_property}
f(0)=\frac{\al+1}{\pi}\int_{\bD}f(z)
R_{n-1}^{(\al,0)}(|z|^2)\,(1-|z|^2)^\al\,\dif\mu(z).
\end{equation}
\end{proposition}

\begin{proof}
For $0<r\le1$, consider
\[
I(r)\eqdef\frac{\al+1}{\pi}\int_{r\bD} f(w)
R_{n-1}^{(\al,0)}(|w|^2)\,(1-|w|^2)^{\al}\,\dif\mu(w).
\]
Then, combining \eqref{aux:poly_expansion} with the polar decomposition, we get that
\[
I(r) = \sum_{k=0}^{n-1} \sum_{j=0}^{\infty} \la_{j,k}
\left((\al+1)\int_{0}^r u^{j+k}R_{n-1}^{(\al,0)}(u^2)(1-u^2)^\al\,2u\,\dif u\right)
\left(\frac{1}{2\pi}\int_0^{2\pi}e^{i(j-k)\theta}\dif\theta\right).
\]
The terms in the inner series vanish whenever $j\neq k$.
Setting $t$ instead of $u^2$ in the integral, we have
\[
I(r) = \sum_{k=0}^{n-1}\la_{k,k}
(\al+1)\int_0^{r^2} t^k R_{n-1}^{(\al,0)}(t)(1-t)^{\al}\dif t.
\]
Take limits in both sides when $r$ tends to $1$. 
Apply Lebesgue's dominated convergence theorem
to get the integral over $\bD$ in the left-hand side
and the integral over $[0,1)$ in the right-hand side.
By formula \eqref{eq:R_interval_k_delta}, 
\[
I(1)=\sum_{k=0}^{n-1}\la_{k,k}(\al+1)
\int_0^1 t^k R_{n-1}^{(\al,0)}(t)(1-t)^{\al}\dif t
=\sum_{k=0}^{n-1} \la_{k,k}\,\delta_{k,0}
=\la_{0,0}=f(0). \qedhere
\]
\end{proof}

For $\al=0$, Proposition~\ref{prop:weighted_mean_value_property}
reduces to the following
mean value property that appeared in~\cite{Koshelev1977} and \cite{Pessoa2014}.

\begin{corollary}
Let $z\in\bC$, $r>0$, and $f\in\cA_n(z+r\bD)$ such that
\[
\int_{z+r\bD}|f(w)|\,\dif\mu(w)<+\infty.
\]
Then
\begin{equation}\label{eq:mean_value_property}
f(z)=\frac{\al+1}{\pi\,r^2}\int_{z+r\bD}f(w)
R_{n-1}^{(\al,0)}\left(\frac{|w-z|^2}{r^2}\right)
\dif{}\mu(w).
\end{equation}
\end{corollary}

\begin{proof}
Denote by $\varphi$ the linear change of variables $\varphi(w)\eqdef rw+z$.
If $f\in\cA_n(z+r\bD)$, 
then $f\circ \varphi\in\cA_n(\bD)$. 
Applying \eqref{eq:weighted_mean_value_property} to $f\circ \varphi$, we obtain \eqref{eq:mean_value_property}.
\end{proof}

\subsection*{Weighted Bergman spaces of polyanalytic functions on general complex domains}

Given $n$ in $\bN$,
an open subset $\Om$ of $\bC$
and a continuous function
$W\colon\Om\to(0,+\infty)$,
we denote by $\cA_n^2(\Om,W)$
the space of $n$-analytic functions
belonging to $L^2(\Om,W)$
and provided with the norm of $L^2(\Om,W)$.
The mean value property~\eqref{eq:mean_value_property}
implies that the evaluation functionals
in $\cA_n^2(\Om,W)$ are bounded
(moreover, they are uniformly bounded on compacts),
and $\cA_n^2(\Om,W)$ is a RKHS.
Here are proofs of these facts.

\begin{lemma}\label{lem:evaluation_functionals_on_An_are_bounded}
Let $K$ be a compact subset of $\Om$.
Then there exists a number $C_{n,W,K}>0$
such that for every $f$ in $\cA_n^2(\Om,W)$
and every $z$ in $K$,
\begin{equation}\label{eq:evaluation_functionals_on_An_are_bounded}
|f(z)|\le C_{n,W,K} \|f\|_{\cA_n^2(\Om,W)}.
\end{equation}
\end{lemma}

\begin{proof}
Let $r_1$ be the distance from $K$ to $\bC\setminus\Om$.
Since $K$ is compact and $\bC\setminus\Om$ is closed,
$r_1>0$.
Put $r\eqdef\min\{r_1/2,1\}$,
$K_1\eqdef\{w\in\bC\colon\ d(w,K)\le r\}$,
\[
C_1
\eqdef
\left(\max_{0\le t\le 1}|R_{n-1}^{(\al,0)}(t)|\right)
\left(\max_{w\in K_1}\frac{1}{\sqrt{W(w)}}\right).
\]
For every $z$ in $K$,
we estimate $|f(z)|$ from above
applying~\eqref{eq:mean_value_property}
and Schwarz inequality:
\begin{align*}
|f(z)|
&\le
\frac{1}{\pi r^2}
\int_{z+r\bD}
|f(w)|\,\left|R_{n-1}^{(\al,0)}\left(\frac{|w-z|^2}{r^2}\right)\right|\,\dif\mu(w)
\\
&\le
\frac{C_1}{\pi r^2}
\int_{z+r\bD}|f(w)|\,\sqrt{W(w)}\,\dif\mu(w)
\\
&\le\frac{C_1}{\pi r^2}
\left(\int_{z+r\bD}|f(w)|^2 W(w)\,\dif\mu(w)\right)^{1/2}
\left(\int_{z+r\bD}1\,\dif\mu(w)\right)^{1/2}
\\
&\le\frac{C_1}{\sqrt{\pi r^2}}\,\|f\|_{\cA_n^2(\Om,W)}.
\end{align*}
So, \eqref{eq:evaluation_functionals_on_An_are_bounded}
is fulfilled with $C_{n,W,K}=\frac{C_1}{\sqrt{\pi r^2}}$.
\end{proof}

\begin{proposition}\label{prop:poly_Bergman_is_RKHS}
$\cA_n^2(\Om,W)$ is a RKHS.
\end{proposition}

\begin{proof}
Given a Cauchy sequence in $\cA_n^2(\Om,W)$,
for every compact $K$
it converges uniformly on $K$
by Lemma~\ref{lem:evaluation_functionals_on_An_are_bounded}.
The pointwise limit of this sequence
is also polyanalytic
by~\cite[Corollary~1.8]{Balk1991},
and it coincides a.e. with the limit
in $L^2(\Om,W)$.
Lemma~\ref{lem:evaluation_functionals_on_An_are_bounded}
also assures the boundedness of the evalution functionals and thereby the existence of the reproducing kernel.
See similar proofs in~\cite[Proposition~3.3]{MaximenkoTelleriaRomero2020}.
\end{proof}

Denote by $\cA_{(n)}^2(\Om,W)$
the orthogonal complement of $\cA_{n-1}^2(\Om,W)$ in $\cA_n^2(\Om,W)$.

\begin{corollary}\label{cor:true_poly_is_RKHS}
$\cA_{(n)}^2(\Om,W)$ is a RKHS.
\end{corollary}

\section{Weighted Bergman spaces of polyanalytic functions on the unit disk}
\label{sec:spaces}

In the rest of the paper,
we suppose that $n\in\bN$ and $\al>-1$.
Given $z$ in $\bD$,
denote by $K^{(\al)}_{n,z}$
the reproducing kernel of $\cA_n^2(\bD,\mu_\al)$
at the point $z$
and by $K^{(\al)}_{(n),z}$
the reproducing kernel of $\cA_{(n)}^2(\bD,\mu_\al)$
at the point $z$.
Hachadi and Youssfi~\cite{HachadiYoussfi2019}
computed the reproducing kernel of $\cA_n^2(\bD,\mu_\al)$:
\begin{equation}\label{eq:RK_weighted_poly_Bergman_on_disk}
K^{(\al)}_{n,z}(w)
=\frac{(1-\conjw z)^{n-1}}{(1-\conjz w)^{n+1}}
R_{n-1}^{(\al,0)}\left(\left|\frac{z-w}{1-\conjz w}\right|^2\right).
\end{equation}
Their method uses~\eqref{eq:weighted_mean_value_property}
and a generalization of the unitary operator constructed by Pessoa~\cite{Pessoa2014}.
Formula~\eqref{eq:RK_weighted_poly_Bergman_on_disk}
implies an exact expression for the norm of $K^{(\al)}_{n,z}$, which is also the norm of the evaluation functional at the point $z$:
\begin{equation}\label{eq:norm_of_reproducing_kernel}
\|K^{(\al)}_{n,z}\|
=\sqrt{(n+\al)\binom{n+\al-1}{n-1}}\:
\frac{1}{1-|z|^2}.
\end{equation}
Obviously, the reproducing kernel of $\cA_{(n)}^2(\bD,\mu_\al)$ can be written as
\begin{equation}\label{eq:RK_weighted_true_poly_Bergman_on_disk}
K^{(\al)}_{(n),z}(w)
=K^{(\al)}_{n,z}(w)-K^{(\al)}_{n-1,z}(w).
\end{equation}
Unfortunately, we are unable to obtain
a simpler formula for $K^{(\al)}_{(n),z}$.

\subsection*{Orthonormal basis in $\cA_n^2(\bD,\mu_\al)$}

\begin{proposition}\label{prop:basis_in_An}
The family $(\basic{\al}{p}{q})_{p\in\bNz,q<n}$
is an orthonormal basis of $\cA_n^2(\bD,\mu_\al)$.
\end{proposition}

\begin{proof}
Is clear that the family is contained in $\cA_n^2(\bD,\mu_\al)$,
and by Proposition~\ref{prop:b_is_orthonormal_basis_in_L2} is orthonormal.
Using ideas of
Ramazanov~\cite[proof of Theorem~2]{Ramazanov1999} we will show the total property.
Suppose that $f\in\cA_n^2(\bD,\mu_\al)$
and $\langle f,\basic{\al}{p}{q}\rangle=0$
for every $p$ in $\bNz$ and $q<n$.
For $r>0$,
using expansion~\eqref{aux:poly_expansion}
and the orthogonality of the Fourier basis on $\bT$, we easily obtain
\[
\int_{r\bD} f\,\overline{\basic{\al}{p}{q}}\,\dif\mu_\al
=\sum_{k=0}^{n-1} \la_{k+p-q,k}
\int_{r\bD} m_{k+p-q,k} \overline{\basic{\al}{p}{q}}\,\dif\mu_\al.
\]
The dominated convergence theorem
allows us to pass to integrals over $\bD$, because $f\,\overline{\basic{\al}{p}{q}}$ and $m_{k+p-q,k}\,\overline{\basic{\al}{p}{q}}$
belong to $L^1(\bD,\mu_\al)$.
Now the assumption $f\perp \basic{\al}{p}{q}=0$ yields
\begin{equation}\label{eq:system_of_equations_on_coefs}
\sum_{k=0}^{n-1} 
\langle m_{k+p-q,k}, \basic{\al}{p}{q}\rangle
\la_{k+p-q,k}
=0
\qquad(p\in\bNz,\ 0\le q<n).
\end{equation}
For a fixed $\xi$ in $\bZ$
with $\xi>-n$,
put $s=\min\{n,n+\xi\}$.
The vector
$[\la_{k+\xi,k}]_{k=\max\{0,-\xi\}}^{n-1}$ satisfies
the homogeneous linear system~\eqref{eq:system_of_equations_on_coefs}
with the $s\times s$ matrix
\[
\left[
\langle
m_{\xi+k,k},\basic{\al}{\xi+q}{q}
\rangle
\right]_{q,k=\max\{0,-\xi\}}^{n-1}.
\]
By \eqref{eq:product_b_m},
this is an upper triangular matrix
with nonzero diagonal entries,
hence the unique solution of~\eqref{eq:system_of_equations_on_coefs} is zero.
\end{proof}

\begin{corollary}\label{cor:basis_in_true_polyanalytic}
The sequence $(\basic{\al}{p}{n-1})_{p\in\bNz}$
is an orthonormal basis of $\cA_{(n)}^2(\bD,\mu_\al)$.
\end{corollary}

We denote by $P_n^{(\al)}$ and $P_{(n)}^{(\al)}$
the orthogonal projections
acting in $L^2(\bD,\mu_\al)$,
whose images are $\cA_n^2(\bD,\mu_\al)$ and $\cA_{(n)}^2(\bD,\mu_\al)$, respectively.
They can be computed in terms
of the corresponding reproducing kernels:
\[
(P_n^{(\al)} f)(z)
=\langle f,K^{(\al)}_{n,z} \rangle,\qquad
(P_{(n)}^{(\al)}f)(z)
=\langle f,K^{(\al)}_{(n),z} \rangle.
\]
For example,
$(\basic{\al}{p}{q})_{p\in\bNz,q<4}$
is an orthonormal basis of $\cA_{4}^2(\bD,\mu_\al)$,
and $(\basic{\al}{p}{3})_{p\in\bNz}$
is an orthonormal basis of $\cA_{(4)}^2(\bD,\mu_\al)$.
\[
\begin{array}{cccccc}
\cellcolor{lightblue} \basic{\al}{0}{0} &
\cellcolor{lightblue} \basic{\al}{0}{1} &
\cellcolor{lightblue} \basic{\al}{0}{2} &
\cellcolor{lightblue} \basic{\al}{0}{3} &
\basic{\al}{0}{4} & \ldots
\\[1ex]
\cellcolor{lightblue} \basic{\al}{1}{0} &
\cellcolor{lightblue} \basic{\al}{1}{1} &
\cellcolor{lightblue} \basic{\al}{1}{2} &
\cellcolor{lightblue} \basic{\al}{1}{3} &
\basic{\al}{1}{4} & \ldots
\\[1ex]
\cellcolor{lightblue} \basic{\al}{2}{0} &
\cellcolor{lightblue} \basic{\al}{2}{1} &
\cellcolor{lightblue} \basic{\al}{2}{2} &
\cellcolor{lightblue} \basic{\al}{2}{3} &
\basic{\al}{2}{4} & \ldots
\\[1ex]
\cellcolor{lightblue} \basic{\al}{3}{0} &
\cellcolor{lightblue} \basic{\al}{3}{1} &
\cellcolor{lightblue} \basic{\al}{3}{2} &
\cellcolor{lightblue} \basic{\al}{3}{3} &
\basic{\al}{3}{4} & \ldots
\\[1ex]
\cellcolor{lightblue} \vdots &
\cellcolor{lightblue} \vdots &
\cellcolor{lightblue} \vdots &
\cellcolor{lightblue} \vdots &
\vdots & \ddots
\end{array}
\qquad\qquad
\begin{array}{cccccc}
\basic{\al}{0}{0} & \basic{\al}{0}{1} & \basic{\al}{0}{2} &
\cellcolor{pink!60!} \basic{\al}{0}{3} & \basic{\al}{0}{4} & \ldots
\\[1ex]
\basic{\al}{1}{0} & \basic{\al}{1}{1} & \basic{\al}{1}{2} &
\cellcolor{pink!60!} \basic{\al}{1}{3} & \basic{\al}{1}{4} & \ldots
\\[1ex]
\basic{\al}{2}{0} & \basic{\al}{2}{1} & \basic{\al}{2}{2} &
\cellcolor{pink!60!} \basic{\al}{2}{3} & \basic{\al}{2}{4} & \ldots
\\[1ex]
\basic{\al}{3}{0} & \basic{\al}{3}{1} & \basic{\al}{3}{2} &
\cellcolor{pink!60!} \basic{\al}{3}{3} & \basic{\al}{3}{4} & \ldots
\\[1ex]
\vdots & \vdots & \vdots &
\cellcolor{pink!60!} \vdots & \vdots & \ddots
\end{array}
\]

\subsection*{Decomposition of $\boldsymbol{\cA_n^2(\bD,\mu_\al)}$
into subspaces corresponding to different ``frequences''}

We will use the following elementary fact about orthonormal bases in Hilbert spaces.

\begin{proposition}\label{prop:basis_of_intersection}
Let $H_1$ be a Hilbert space and $\cB_1\subseteq H_1$ be an orthonormal basis of $H_1$
(in this proposition we treat orthonormal bases like sets rather than families).
Suppose that $\cB_2$ and $\cB_3$ are some subsets of $\cB_1$. Denote by $H_2$ and $H_3$ the closed subspaces of $H_1$ generated by $\cB_2$ and $\cB_3$, respectively.
Then $\cB_2\cap\cB_3$ is an orthonormal basis of $H_2\cap H_3$.
\end{proposition}

Applying Proposition~\ref{prop:basis_of_intersection}
to the Hilbert space $L^2(\bD,\mu_\al)$
and thinking in terms of orthonormal bases
(see Propositions~\ref{prop:b_is_orthonormal_basis_in_L2}, \ref{prop:basis_in_An},
and Corollaries~\ref{cor:basis_of_truncated_FreqSubspace}, \ref{cor:basis_of_FreqSubspace}),
we easily find the intersection of $\FreqSubspace_\xi^{(\al)}$ and $\cA_n^2(\bD,\mu_\al)$:
\begin{equation}\label{eq:truncated_freqsubspace_as_intersection}
\FreqSubspace_\xi^{(\al)}
\cap\cA_n^{2}(\bD,\mu_\al)
=\begin{cases}
\FreqSubspace_{\xi,\min\{n,n+\xi\}}^{(\al)}, & \xi\ge -n+1; \\
\{0\}, & \xi<-n+1.
\end{cases}
\end{equation}
Here is a description of the subspaces $\FreqSubspace^{(\al)}_{\xi,m}$
in terms of the polar coordinates.

\begin{proposition}\label{prop:truncated_freqsubspaces_in_polar_coordinates}
For every $\xi$ in $\bZ$ and every $s$ in $\bN$,
the space $\FreqSubspace_{\xi,s}^{(\al)}$
consists of all functions of the form
\[
f(r\tau)
= \tau^\xi r^{|\xi|} Q(r^2)\qquad
(r\ge0,\ \tau\in\bT),
\]
where $Q$ is a polynomial of degree $\le s-1$.
Moreover,
\[
\|f\|=\|Q\|_{L^2([0,1),(\al+1)(1-t)^{\al}t^{|\xi|}\,\dif{}t)}.
\]
\end{proposition}

\begin{proof}
The result follows directly by Proposition~\eqref{prop:polar_description_of_W}
and formula~\eqref{polar_description_of_D2}.
\end{proof}

The decomposition of $\cA_n^{2}(\bD,\mu_\al)$ into a direct sum of
the ``truncated frequency subspaces''
shown below follows
from Proposition~\ref{prop:basis_in_An}
and Corollary~\ref{cor:basis_of_truncated_FreqSubspace},
and plays a crucial role in the study of radial operators.
It can be seen as the ``Fourier series decomposition'' of $\cA_n^2(\bD,\mu_\al)$.

\begin{proposition}\label{prop:decomposition_An_into_truncated_freqsubspaces}
\begin{equation}\label{eq:decomposition_An_into_truncated_freqsubspaces}
\cA_n^{2}(\bD,\mu_\al)
=\bigoplus_{\xi=-n+1}^\infty \FreqSubspace^{(\al)}_{\xi,\min\{n,n+\xi\}}.
\end{equation}
\end{proposition}

Let us illustrate Proposition~\ref{prop:decomposition_An_into_truncated_freqsubspaces} for $n=3$ with a table (we have marked in different shades of blue 
the basic functions that generate each truncated diagonal):
\[
\begin{array}{ccccc}
\medstrut
\cellcolor{blue!32} \basic{\al}{0}{0} &
\cellcolor{blue!24} \basic{\al}{0}{1} &
\cellcolor{blue!16} \basic{\al}{0}{2} &
\basic{\al}{0}{3} & \ldots
\\
\medstrut
\cellcolor{blue!40} \basic{\al}{1}{0} &
\cellcolor{blue!32} \basic{\al}{1}{1} &
\cellcolor{blue!24} \basic{\al}{1}{2} &
\basic{\al}{1}{3} & \ldots
\\
\medstrut
\cellcolor{blue!48} \basic{\al}{2}{0} &
\cellcolor{blue!40} \basic{\al}{2}{1} &
\cellcolor{blue!32} \basic{\al}{2}{2} &
\basic{\al}{2}{3} & \ldots
\\
\medstrut
\cellcolor{blue!56} \basic{\al}{3}{0} &
\cellcolor{blue!48} \basic{\al}{3}{1} &
\cellcolor{blue!40} \basic{\al}{3}{2} &
\basic{\al}{3}{3} & \ldots
\\
\medstrut
\vdots & \vdots & \vdots & \vdots & \ddots
\end{array}
\]
Define $U_n^{(\al)}\colon\cA_n^2(\bD,\mu_\al)\to\bigoplus_{\xi=-n+1}^\infty\bC^{\min\{n,n+\xi\}}$,
\begin{equation}\label{eq:Un_definition}
(U_n^{(\al)} f)_{\xi,q}\eqdef
\langle f, \basic{\al}{q+\xi}{q}\rangle
\qquad(\xi\ge -n+1,\
\max\{0,-\xi\}\le q\le n-1).
\end{equation}
Here, for $-n+1\le \xi<0$,
the componentes of vectors in $\bC^{n+\xi}$
are enumerated from $-\xi$ to $n-1$.

\begin{proposition}\label{prop:An_to_vector_sequence}
The operator $U_n^{(\al)}$ is an isometric isomorphism of Hilbert spaces.
\end{proposition}

\begin{proof}
Follows from Proposition~\ref{prop:basis_in_An} or, even easier, from Proposition~\ref{prop:decomposition_An_into_truncated_freqsubspaces} and the fact that $(\basic{\al}{q+\xi}{q})_{q=\max\{0,-\xi\}}^{n-1}$
is an orthonormal basis of
$\FreqSubspace^{(\al)}_{\xi,\min\{n,n+\xi\}}$
(see Corollary~\ref{cor:basis_of_truncated_FreqSubspace}).
\end{proof}

An analog of the upcoming fact 
for the unweighted poly-Bergman space
was proved by Vasilevski~\cite[Section 4.2]{Vasilevski2008book}.
We obtain it as a corollary from 
Proposition~\ref{prop:b_is_orthonormal_basis_in_L2} 
and Corollary~\ref{cor:basis_in_true_polyanalytic}.

\begin{corollary}\label{cor:decomposition_L2_into_true_polyanalytic}
The space $L^2(\bD,\mu_\al)$
is the orthogonal sum
of the subspaces $\cA_{(m)}^2(\bD,\mu_\al)$, $m\in\bN$:
\[
L^2(\bD,\mu_\al)=\bigoplus_{m\in\bN}\cA_{(m)}^2(\bD,\mu_\al).
\]
\end{corollary}

\section{The set of Toeplitz operators is not weakly dense}
\label{sec:not_weakly_dense}

Given a Hilbert space $H$, we denote by $\cB(H)$
the algebra of all bounded operators acting in $H$.
If $H$ is a RKHS naturally embedded into $L^2(\bD,\mu_\al)$
and $S\in\cB(H)$,
then the \emph{Berezin transform} of $S$ is defined by
\[
\Berezin_H(S)(z)
\eqdef
\frac{\langle SK_z,K_z\rangle_H}{\langle K_z,K_z\rangle_H},
\qquad\text{i.e.},\qquad
\Berezin_H(S)(z)
=\frac{(SK_z)(z)}{K_z(z)}.
\]
The Berezin transform can be considered as a
bounded linear operator $\cB(H)\to L^\infty(\Omega)$.
Stroethoff proved~\cite{Stroethoff1997}
that $\Berezin_H$ is injective for various RKHS
of analytic functions, in particular, for $H=\cA^2_1(\bD)$.
Engli\v{s} noticed \cite[Section~2]{Englis2006} that
$\Berezin_H$ is not injective
for various RKHS of harmonic functions.
The idea of Engli\v{s} can be applied without any changes to various spaces of polyanalytic and polyharmonic functions.
For clarity of presentation,
we state the result of Engli\v{s}
for $\cA^2_n(\bD,\mu_\al)$, $n\ge 2$,
and repeat his proof.

\begin{proposition}\label{prop:Berezin_not_injective}
Let $H=\cA^2_n(\bD,\mu_\al)$ with $n\ge 2$.
Then the Berezin transform $\Berezin_H$ is not injective.
\end{proposition}

\begin{proof}
Let $f\in H$ such that $\overline{f}\in H$
and the functions
$f,\overline{f}$ are linearly independent.
For example, $f(z)\eqdef z$.
Following the idea from~\cite[Section~2]{Englis2006},
consider the operator
\[
Sh\eqdef\langle h,f\rangle_H f
- \langle h,\overline{f}\rangle_H\,\overline{f}.
\]
Then $S\ne0$, but
$\langle SK_z,K_z\rangle_H = |f(z)|^2-|f(z)|^2=0$
for every $z$ in $\bD$.
So, $\Berezin_H(S)$ is the zero constant.
\end{proof}

Given a function $g$ in $L^\infty(\bD)$,
let $M_g$ be the multiplication operator
defined on $L^2(\bD,\mu_\al)$ by $M_g f\eqdef gf$.
If $H$ is a closed subspace of $L^2(\bD,\mu_\al)$,
then the \emph{Toeplitz operator} $T_{H,g}$ is defined on $H$ by
\[
T_{H,g}(f)\eqdef P_H(gf)=P_H M_g f.
\]
For $H=\cA_n^2(\bD,\mu_\al)$ and $H=\cA^2_{(n)}(\bD,\mu_\al)$,
we write just $T^{(\al)}_{n,g}$ and $T^{(\al)}_{(n),g}$, respectively.
The proof of the following fact is the same as the proof of
\cite[Proposition 3.18]{MaximenkoTelleriaRomero2020}
or the proof of \cite[Theorem~4]{BergerCoburn1986}.

\begin{proposition}\label{prop:Toeplitz_injective}
If $g\in L^\infty(\bD)$
and $T^{(\al)}_{n,g}=0$, then $g=0$ a.e.
In other words, the function
$g\mapsto T^{(\al)}_{n,g}$, acting from
$L^\infty(\bD)$ to $\cB(\cA_n^2(\bD,\mu_\al))$,
is injective.
\end{proposition}

Inspired by the idea of Engli\v{s}
explained in the proof of Proposition~\ref{prop:Berezin_not_injective},
we will prove that set of Toeplitz operators
is not weakly dense in $\cB(\cA^2_n(\bD,\mu_\al))$ with $n\ge 2$.
First, let us prove an auxiliary fact from linear algebra:
bounded quadratic forms separate linearly independent vectors.

\begin{lemma}
\label{lem:quadratic_forms_separate_linindep_vectors}
Let $H$ be a Hilbert space
and $f,g$ be two linearly independent vectors in $H$.
Then there exists $S$ in $\cB(H)$ such that
\[
\langle Sf,f\rangle_H \ne \langle Sg,g\rangle_H.
\]
\end{lemma}

\begin{proof}
Without lost of generality,
we will suppose that $\|f\|_H=1$.
Decompose $g$ into the linear combination
$g=\la_1 f+\la_2 h$,
with $\la_1,\la_2\in\bC$, $\|h\|_H=1$, $h\perp f$.
More explicitly,
\[
\la_1\eqdef\langle g,f\rangle_H,\qquad
w\eqdef g-\la_1 f,\qquad
\la_2\eqdef\|w\|_H,\qquad
h\eqdef\frac{1}{\la_2}w.
\]
Define $S$ as the orthogonal projection onto $h$:
\[
Sv\eqdef \langle v,h\rangle_H h\qquad(v\in H).
\]
Then $Sf=0$ and $Sg=\la_2 h$, therefore
$\langle Sf,f\rangle_H = 0$ and
$\langle Sg,g\rangle_H = \la_2^2 > 0$.
\end{proof}

\begin{theorem}\label{thm:not_dense}
Let $H=\cA^2_n(\bD,\mu_\al)$ with $n\ge 2$.
Then the set of the Toeplitz operators
with bounded symbols is not weakly dense in $\cB(H)$.
\end{theorem}

\begin{proof}
Let $f\in H$ such that $\overline{f}\in H$
and the functions
$f,\overline{f}$ are linearly independent.
For example, $f(z)\eqdef z$.
The set
\[
W\eqdef\{S\in\cB(H)\colon\ 
\langle Sf,f\rangle_H
=\langle S\overline{f},\overline{f}\rangle_H\}
\]
is a weakly closed subspace of $\cB(H)$.
By Lemma~\ref{lem:quadratic_forms_separate_linindep_vectors},
$W\ne\cB(H)$.
On the other hand, for every $a$ in $L^\infty(\bD)$
\[
\langle T^{(\al)}_{n,a} f,f\rangle_H
= \int_X a\,|f|^2\,\dif\mu_\al
= \langle T^{(\al)}_{n,a}\, \overline{f},\,\overline{f}\rangle_H,
\]
i.e., $\{T^{(\al)}_{n,a}\colon\ a\in L^\infty(\bD)\}\subseteq W$.
\end{proof}

\begin{remark}
An analog of Theorem~\ref{thm:not_dense} is true for the space of $\mu_\al$-square integrable
$n$-harmonic functions on $\bD$, with $n\ge 1$.
\end{remark}

\section{Von Neumann algebras of radial operators}
\label{sec:radial}

\subsection*{Set of operators diagonalized by a family of subspaces}

The theory of von Neumann algebras and their decompositions is well developed.
For our purposes, it is sufficient to use
the following elementary scheme
from~\cite{MaximenkoTelleriaRomero2020}.
This scheme is similar to ideas from
\cite{Zorboska2003,GrudskyMaximenkoVasilevski2013,Quiroga2016}.

\begin{definition}\label{def:diagonalizing_family_of_subspaces}
Let $H$ be a Hilbert space,
$\cU$ be a self-adjoint subset of $\cB(H)$,
and $(W_j)_{j\in J}$ be a finite or countable family
of nonzero closed subspaces of $H$ such that
$H=\bigoplus_{j\in J}W_j$.
We say that this family \emph{diagonalizes} $\cU$ if
the following two conditions are satisfied.
\begin{enumerate}
\item For each $j$ in $J$ and each $U$ in $\cU$,
there exists $\la_{U,j}$ in $\bC$ such that
$W_j\subseteq\ker(\la_{U,j}I-U)$,
i.e., $U(v)=\la_{U,j}v$ for every $v$ in $W_j$.
\item For every $j$, $k$ in $J$ with $j\ne k$,
there exists $U$ in $\cU$ such that $\la_{U,j}\ne\la_{U,k}$.
\end{enumerate}
\end{definition}

\begin{proposition}\label{prop:diagonalizing_family_of_subspaces}
Let $H$, $\cU$, and $(W_j)_{j\in J}$ be like in
Definition~\ref{def:diagonalizing_family_of_subspaces}.
Denote by $\cA$ the commutant of $\cU$.
Then $\cA$ consists of all bounded
linear operators that act invariantly on
each of the subspaces $W_j$, with $j\in J$:
\begin{equation}\label{eq:conmutant_description}
\cA=\{S\in\cB(H)\colon\quad\forall j\in J\quad S(W_j)\subseteq W_j\}.
\end{equation}
Furthermore,
$\cA$ is isometrically isomorphic to
$\bigoplus_{j\in J}\cB(W_j)$,
and the von Neumann algebra generated by $\cU$
is isometrically isomorphic to
$\bigoplus_{j\in J}\bC I_{W_j}$.
\end{proposition}

\begin{example}
\label{example:finite_rank}
Let $j_1,\ldots,j_m\in J$,
$\la_1,\ldots,\la_m\in\bC$,
and $u_{j_k},v_{j_k}\in W_{j_k}$
for every $k$ in $\{1,\ldots,m\}$.
Then the operator $S\colon H\to H$ defined by
\begin{equation}\label{eq:finite_rank}
Sf \eqdef \sum_{k=1}^m \la_k
\langle f,u_{j_k}\rangle v_{j_k},
\end{equation}
belongs to $\cA$.
Moreover, every operator of finite rank,
belonging to $\cA$,
can be written in this form.
See the proof of~\cite[Corollary~5.7]{MaximenkoTelleriaRomero2020} for a similar situation.
\end{example}

\begin{proposition}\label{prop:diagonalizing_family_in_subspace}
Let $H$, $\cU$, and $(W_j)_{j\in J}$ be like in
Definition~\ref{def:diagonalizing_family_of_subspaces},
and $H_1$ be a closed subspace of $H$ invariant under $\cU$.
For every $U$ in $\cU$,
denote by $U|_{H_1}^{H_1}$
the compression of $U$
onto the invariant subspace $H_1$, and put
\[
\cU_1\eqdef
\left\{U|_{H_1}^{H_1}\colon\ U\in\cU\right\},\qquad
J_1\eqdef\{j\in J\colon\ W_j\cap H_1\ne\{0\}\}.
\]
Then
\begin{equation}\label{eq:H1_orthogonal_decomposition}
H_1=\bigoplus_{j\in J_1}(W_j\cap H_1),
\end{equation}
and the family $(W_j\cap H_1)_{j\in J}$ diagonalizes $\cU_1$.
\end{proposition}

\begin{example}
\label{example:finite_rank_in_the_subspace}
The operators of finite rank,
commuting with $U|_{H_1}^{H_1}$
for every $U$ in $\cU$,
are of the form~\eqref{eq:finite_rank},
but with $u_{j_k},v_{j_k}\in W_{j_k}\cap H_1$.
\end{example}

\subsection*{\texorpdfstring{Radial operators in $\boldsymbol{L^2(\bD,\mu_\al)}$}{Radial operators in L2(D,mualpha)}}

For each $\tau$ in $\bT$,
we denote by $\rho^{(\al)}(\tau)$
the rotation operator acting in $L^2(\bD,\mu_\al)$ by the rule
\begin{equation}\label{eq:def_rotation}
(\rho^{(\al)}(\tau)f)(z)
\eqdef f(\tau^{-1}z).
\end{equation}
It is easy to see that
$\rho^{(\al)}(\tau_1\tau_2)
=\rho^{(\al)}(\tau_1)\rho^{(\al)}(\tau_2)$,
the operators $\rho^{(\al)}(\tau)$ are unitary,
and for every $f$ in $L^2(\bD,\mu_\al)$
the mapping $\tau\mapsto\rho^{(\al)}(\tau)f$
is continuous (this is easy to check
first for the case when $f$ is a continuous function with compact support).
So,
$(\rho^{(\al)},L^2(\bD,\mu_\al))$
is a unitary representation of the group $\bT$.
The operators commuting with $\rho^{(\al)}(\tau)$
for every $\tau$ in $\bT$
are called \emph{radial operators}.
We denote the set of all radial operators
in $L^2(\bD,\mu_\al)$ by $\cR^{(\al)}$:
\[
\cR^{(\al)}
\eqdef\{\rho^{(\al)}(\tau)\colon\ \tau\in\bT\}'
=\{S\in\cB(L^2(\bD,\mu_\al))\colon\quad
\forall\tau\in\bT\quad\rho^{(\al)}(\tau)S
=S\rho^{(\al)}(\tau)\}.
\]
Since $\{\rho^{(\al)}(\tau)\colon \tau\in\bT\}$
is an autoadjoint subset of $\cB(L^2(\bD,\mu_\al))$,
its commutant $\cR^{(\al)}$ is a von Neumann algebra~\cite{Sakai1971}.

Recall that the subspaces $\FreqSubspace_\xi^{(\al)}$
are defined by~\eqref{def:fiber_subspace}.

\begin{lemma}\label{lem:freqsubspaces_diagonalize_rotation_operators}
The family $(\FreqSubspace_\xi^{(\al)})_{\xi\in\bZ}$ diagonalizes the collection
$\{\rho^{(\al)}(\tau)\colon\ \tau\in\bT\}$
in the sense of
Definition~\ref{def:diagonalizing_family_of_subspaces}.
\end{lemma}

\begin{proof}
1. Let $\tau\in\bT$.
For every $p,q\in\bZ$ with $p-q=\xi$,
formula~\eqref{eq:b_via_Q} implies
\begin{equation}\label{eq:rotation_on_the_basics}
\rho^{(\al)}(\tau) \basic{\al}{p}{q}
= \tau^{q-p} \basic{\al}{p}{q}
=\tau^{-\xi} \basic{\al}{p}{q},
\end{equation}
i.e., $\basic{\al}{p}{q}\in\ker(\tau^{-\xi}I-\rho^{(\al)}(\tau))$.
By~Corollary~\ref{cor:basis_of_FreqSubspace},
the functions $\basic{\al}{p}{q}$
with $p-q=\xi$
form an orthonormal basis of $\FreqSubspace_\xi^{(\al)}$.
So,
\begin{equation}\label{eq:freqsubspaces_subspaces_as_eigensubspaces}
\FreqSubspace_\xi^{(\al)}
\subseteq
\ker(\tau^{-\xi} I-\rho^{(\al)}(\tau)).
\end{equation}
2. Let $\xi_1,\xi_2\in\bZ$ and $\xi_1\ne\xi_2$.
Put
$\tau=\exp\frac{\imagunit\pi}{\xi_1-\xi_2}$.
Then $\tau^{-\xi_1}\ne\tau^{-\xi_2}$.
\end{proof}

\begin{proposition}\label{prop:radial_operators_on_L2}
The von Neumann algebra $\cR^{(\al)}$
consists of all operators that act invariantly on $\FreqSubspace_\xi^{(\al)}$
for every $\xi$ in $\bZ$,
and is isometrically isomorphic to $\bigoplus_{\xi\in\bZ}\cB(\FreqSubspace_\xi^{(\al)})$.
\end{proposition}

\begin{proof}
Follows from
Proposition~\ref{prop:diagonalizing_family_of_subspaces}
and Lemma~\ref{lem:freqsubspaces_diagonalize_rotation_operators}.
\end{proof}

The \emph{radialization transform}
$\Radialization^{(\al)}\colon\cB(L^2(\bD,\mu_\al))
\to\cB(L^2(\bD,\mu_\al))$,
introduced by Zorboska~\cite{Zorboska2003},
acts by the rule
\[
\Radialization^{(\al)}(S)
\eqdef \int_\bT \rho(\tau) S \rho(\tau^{-1}) \,\dif\mu_\bT(\tau),
\]
where $\mu_\bT$ is the normalized Haar measure on $\bT$, and the integral is understood in the weak sense.
The condition $S\in\cR^{(\al)}$
is equivalent to $\Radialization^{(\al)}(S)=S$.

\subsection*{\texorpdfstring{Radial operators in $\boldsymbol{\cA_n^2(\bD,\mu_\al)}$}{Radial operators in An2(D,mualpha)}}

\begin{proposition}\label{prop:polyBergman_is_rotation_invariant}
The space $\cA_n^2(\bD,\mu_\al)$
is invariant under every rotation $\rho^{(\al)}(\tau)$, $\tau\in\bT$.
\end{proposition}

\begin{proof}[First proof]
The reproducing kernel of $\cA_n^2(\bD,\mu_\al)$,
given by \eqref{eq:RK_weighted_poly_Bergman_on_disk},
is invariant under simultaneous rotations in both arguments:
\begin{equation}\label{eq:Kn_and_rotations}
K^{(\al)}_{n,\tau z}(\tau w)
=K^{(\al)}_{n,z}(w)\qquad
(z,w\in\bD,\ \tau\in\bT).
\end{equation}
According to~\cite[Proposition~4]{MaximenkoTelleriaRomero2020},
this implies the invariance of the subspace.
\end{proof}

\begin{proof}[Second proof]
By~\eqref{eq:def_rotation},
the elements of the orthonormal basis~$(\basic{\al}{p}{q})_{p\in\bNz,0\le q<n}$ are eigenfunctions of~$\rho^{(\al)}$.
\end{proof}

For every $\tau$ in $\bT$,
we denote by $\rho_n^{(\al)}(\tau)$
the compression of $\rho^{(\al)}(\tau)$
onto the space $\cA_n^2(\bD,\mu_\al)$.
In other words,
the operator $\rho_n^{(\al)}(\tau)$
acts in $\cA_n^2(\bD,\mu_\al)$
and is defined by \eqref{eq:def_rotation}.
So, $(\rho_n^{(\al)},\cA_n^2(\bD,\mu_\al))$
is a unitary representation of $\bT$.
We denote by $\cR_n^{(\al)}$
the commutant of this representation,
i.e., the von Neumann algebra
of all bounded linear radial operators
acting in $\cA_n^2(\bD,\mu_\al)$.

Denote by $\fM_n$ the following direct sum of matrix algebras:
\[
\fM_n \eqdef \bigoplus_{\xi=-n+1}^\infty \Mat_{\min\{n,n+\xi\}}
= \left(\bigoplus_{\xi=-n+1}^{-1} \Mat_{n+\xi}\right)
\oplus
\left(\bigoplus_{\xi=0}^{\infty} \Mat_n\right).
\]
For example,
\[
\fM_3
=\underbrace{\Mat_1}_{\xi=-2}
\oplus\underbrace{\Mat_2}_{\xi=-1}
\oplus\underbrace{\Mat_3}_{\xi=0}
\oplus\underbrace{\Mat_3}_{\xi=1}
\oplus\underbrace{\Mat_3}_{\xi=2}
\oplus\dots.
\]
According to the definition of the direct sum (see~\cite[Definition 1.1.5]{Sakai1971}),
$\fM_n$ consists of all matrix sequences
of the form
$A=(A_\xi)_{\xi=-n+1}^\infty$,
where $A_\xi\in\Mat_{n+\xi}$ if $\xi<0$, $A_\xi\in\Mat_n$ if $\xi\ge0$, and
\[
\sup_{\xi\ge-n+1}\|A_\xi\|<+\infty.
\]
Being a direct sum of W*-algebras,
$\fM_n$ is a W*-algebra.
We identify the elements of $\fM_n$
with the bounded linear operators acting in $\bigoplus_{\xi=-n+1}^\infty\bC^{\min\{n,n+\xi\}}$.
Now we are ready to describe
the structure of $\cR_n^{(\al)}$.
Recall that $U_n^{(\al)}$ is given by~\eqref{eq:Un_definition}.

\begin{theorem}\label{thm:radial_polyanalytic_Bergman}
Let $n\in\bN$.
Then $\cR_n^{(\al)}$
consists of all operators belonging to $\cB(\cA_n^2(\bD,\mu_\al))$
that act invariantly on each subspace
$\FreqSubspace^{(\al)}_{\xi,\min\{n,n+\xi\}}$, for $\xi\ge-n+1$.
Furthermore,
\[
\cR_n^{(\al)}
\cong\bigoplus_{\xi=-n+1}^\infty \cB(\FreqSubspace^{(\al)}_{\xi,\min\{n,n+\xi\}}),
\]
and $\cR_n^{(\al)}$ is spatially isomorphic to $\fM_n$:
\[
U_n^{(\al)} \cR_n^{(\al)} (U_n^{(\al)})^\ast
=\fM_n.
\]
\end{theorem}

\begin{proof}
We apply the scheme from Propositions~\ref{prop:diagonalizing_family_of_subspaces},
\ref{prop:diagonalizing_family_in_subspace},
\[
W_j=\FreqSubspace_\xi^{(\al)},\qquad
\cU=\{\rho^{(\al)}(\tau)\colon\ \tau\in\bT\},
\]
and $H_1=\cA_n^2(\bD,\mu_\al)$.
By~\eqref{eq:truncated_freqsubspace_as_intersection}, we obtain
\[
J_1=\{\xi\in\bZ\colon\ \xi\ge -n+1\},\qquad
\cA_n^2(\bD,\mu_\al)\cap \FreqSubspace_\xi^{(\al)}
=
\FreqSubspace^{(\al)}_{\xi,\min\{n,n+\xi\}}.
\]
So, the W*-algebra $\cR_n^{(\al)}$
is isometrically isomorphic to the direct sum
of $\cB(\FreqSubspace^{(\al)}_{\xi,\min\{n,n+\xi\}})$, with $\xi\ge-n+1$.
Using the orthonormal basis $(\basic{\al}{\xi+k}{k})_{k=\max\{0,-\xi\}}^{n-1}$
of $\FreqSubspace^{(\al)}_{\xi,\min\{n,n+\xi\}}$,
we represent linear operators
on this space as matrices.
Define $\Phi_n^{(\al)}\colon\cR_n^{(\al)}\to\fM_n$ by
\begin{equation}\label{eq:Phi}
\Phi_n^{(\al)}(S)
\eqdef
\left(\left[\left\langle S\basic{\al}{\xi+k}{k},\basic{\al}{\xi+j}{j}\right\rangle\right]_{j,k=\max\{0,-\xi\}}^{n-1}
\right)_{\xi=-n+1}^{\infty}.
\end{equation}
In other words, $\Phi_n^{(\al)}(S)=U_n^{(\al)} S (U_n^{(\al)})^\ast$, i.e., $\Phi_n^{(\al)}$ is an isometrical isomorphism of W*-algebras induced by the unitary operator $U_n^{(\al)}$.
\end{proof}

Radial operators of finite rank, acting in $\cA_n^2(\bD,\mu_\al)$, can be constructed as in Examples~\ref{example:finite_rank} and \ref{example:finite_rank_in_the_subspace}.

It is easy to verify
(see a more general result in
\cite[Corollary~4.3]{MaximenkoTelleriaRomero2020})
that if $\cA_n^2=\cA_n^2(\bD,\mu_\al)$
and $S\in\cR_n^{(\al)}$,
then $\Berezin_{\cA_n^2}(S)$
is a radial function.
For $n=1$, the Berezin transform $\Berezin_{\cA_1^2}$ is injective.
So, if $S\in\cB(\cA_1^2(\bD,\mu_\al))$
and the function $\Berezin_{\cA_1^2}(S)$
is radial, then the operator $S$ is radial.

\subsection*{\texorpdfstring{Radial operators in $\boldsymbol{\cA_{(n)}^2(\bD,\mu_\al)}$}{Radial operators in An2(D,mualpha)}}

Let $n\in\bN$.
The space $\cA^2_{(n)}(\bD,\mu_\al)$
is invariant under
the rotation $\rho^{(\al)}(\tau)$
for all $\tau$ in $\bT$.
The proof is similar to the proof of Proposition~\ref{prop:polyBergman_is_rotation_invariant}.
Denote the compression of $\rho^{(\al)}(\tau)$ onto $\cA^2_{(n)}(\bD,\mu_\al)$
by $\rho_{(n)}(\tau)$.
Let $\cR^{(\al)}_{(n)}$ be the von Neumann algebra of all radial operators in $\cA_{(n)}^2(\bD,\mu_\al)$.

\begin{theorem}\label{thm:radial_true_polyanalytic_Bergman}
$\cR^{(\al)}_{(n)}$ consists of all operators
belonging to $\cB(\cA_{(n)}^2(\bD,\mu_\al))$
that are diagonal with respect to the orthonormal basis
$(\basic{\al}{p}{n-1})_{p=0}^\infty$.
Furthermore,
\[
\cR^{(\al)}_{(n)}
\cong \ell^\infty(\bNz).
\]
\end{theorem}

\begin{proof}
Corollaries~\ref{cor:basis_of_FreqSubspace} and \ref{cor:basis_in_true_polyanalytic} give
\begin{equation}\label{eq:intersection_diagonal_with_column}
\FreqSubspace_\xi^{(\al)}\cap\cA_{(n)}^2(\bD,\mu_\al)
=\begin{cases}
\bC \basic{\al}{\xi+n-1}{n-1}, & \xi\ge-n+1,\\
\{0\}, & \xi<-n+1.
\end{cases}
\end{equation}
By~Propositions~\ref{prop:diagonalizing_family_of_subspaces},
\ref{prop:diagonalizing_family_in_subspace}
and formula~\eqref{eq:intersection_diagonal_with_column},
$\cR^{(\al)}_{(n)}$ consists of the operators
that act invariantly on
$\bC \basic{\al}{\xi+n-1}{n-1}$, $\xi\ge-n+1$,
i.e., are diagonal with respect to the basis
$(\basic{\al}{p}{n-1})_{p=0}^\infty$.
Therefore the function
$\Phi^{(\al)}_{(n)}\colon\cR^{(\al)}_{(n)}\to\ell^\infty(\bNz)$,
defined by
\begin{equation}\label{eq:Phi_true}
\Phi^{(\al)}_{(n)}(S)
=\bigl(\langle S \basic{\al}{p}{n-1},\basic{\al}{p}{n-1}\rangle\bigr)_{p=0}^\infty,
\end{equation}
is an isometric isomorphism.
\end{proof}

\section{Radial Toeplitz operators in polyanalytic Bergman spaces}
\label{sec:Toeplitz}

This section is similar to \cite[Section~6]{MaximenkoTelleriaRomero2020},
but here we use Jacobi polynomials
instead of the generalized Laguerre polynomials.

\subsection*{Radial functions}

Given $g$ in $L^\infty(\bD)$,
define $\radialization(g)\colon \bD\to\bC$ by
\begin{equation}\label{eq:def_rad}
\radialization(g)(z)
\eqdef\int_\bT g(\tau z)\,\dif\mu_\bT(\tau).
\end{equation}
Given $a$ in $L^\infty([0,1))$,
define $\widetilde{a}\colon\bD\to\bC$ by
\[
\widetilde{a}(z)\eqdef a(|z|)\qquad(z\in\bD).
\]
The proof of the following criterion is a simple exercise.

\begin{proposition}\label{prop:radial_criterion}
Given $g$ in $L^\infty(\bD)$,
the following conditions are equivalent:
\begin{itemize}
\item[(a)] for every $\tau$ in $\bT$,
the equality $g(\tau z)=g(z)$ is true for a.e. $z$ in $\bD$;
\item[(b)] for every $\tau$ in $\bT$,
the equality $\rho^{(\al)}(\tau)g=g$
is true a.e.;
\item[(c)] $\radialization(g)=g$ a.e.;
\item[(d)] there exists $a$ in $L^\infty([0,1))$ such that $g=\widetilde{a}$ a.e.
\end{itemize}
\end{proposition}

\subsection*{Radial multiplication operators in $L^2(\bD,\mu_\al)$}

\begin{proposition}\label{prop:radialization_of_multiplication_operator}
Let $g\in L^\infty(\bD)$.
Then
$\Radialization^{(\al)}(M_g)
=M_{\radialization(g)}^{(\al)}$.
\end{proposition}

Given $a$ in $L^\infty([0,1))$,
we define the numbers $\be_{a,\al,\xi,j,k}$ by
\begin{equation}\label{eq:gamma_as_integral_jac}
\be_{a,\al,\xi,j,k}
\eqdef
\int_0^1 a(\sqrt{t})
\jac_{\min\{j,j+\xi\}}^{(\al,|\xi|)}(t)
\jac_{\min\{k,k+\xi\}}^{(\al,|\xi|)}(t)
\,\dif{}t,
\end{equation}
i.e.,
\begin{equation}\label{eq:gamma_as_integral}
\be_{a,\al,\xi,j,k}
\eqdef
\jaccoef{\al}{|\xi|}{\min\{q+\xi,q\}}
\jaccoef{\al}{|\xi|}{\min\{k+\xi,k\}}
\int_0^1 a(\sqrt{t})
Q_{\min\{q,q+\xi\}}^{(\al,|\xi|)}(t)
Q_{\min\{k,k+\xi\}}^{(\al,|\xi|)}(t)
\,(1-t)^\al\,t^{|\xi|}\,\dif{}t.
\end{equation}

\begin{proposition}\label{prop:radial_multiplication}
Let $a\in L^\infty([0,1))$.
Then $M_{\widetilde{a}}\in\cR^{(\al)}$, and
\begin{equation}\label{eq:radial_mul_operator_on_basis}
\langle M_{\widetilde{a}}\basic{\al}{p}{q},\basic{\al}{j}{k}\rangle
=\langle \widetilde{a} \basic{\al}{p}{q},\basic{\al}{j}{k}\rangle
=\delta_{p-q,j-k} \beta_{a,\al,p-q,q,k}.
\end{equation}
\end{proposition}

\begin{proof}
Since $\widetilde{a}$
is invariant under rotations,
it follows directly from definitions
that $M_{\widetilde{a}}^{(\al)}$
commutes with $\rho^{(\al)}(\tau)$ for every $\tau$.
This is a particular case of \cite[Lemma~4.4]{MaximenkoTelleriaRomero2020}.
Formula~\eqref{eq:radial_mul_operator_on_basis}
is obtained directly using polar coordinates.
\end{proof}

\subsection*{\texorpdfstring{Radial Toeplitz operators
in $\boldsymbol{\cA_n^2(\bD,\mu_\al)}$}{Radial Toeplitz operators in An2(D,mualpha)}}

\begin{proposition}\label{prop:radial_Toeplitz_operator}
Let $g\in L^\infty(\bD)$.
Then $T^{(\al)}_{n,g}$ is radial
if and only if the function $g$ is radial.
\end{proposition}

\begin{proof}
Follows from
Proposition~\ref{prop:Toeplitz_injective}
and \cite[Corollaries~4.6, 4.7]{MaximenkoTelleriaRomero2020}.
\end{proof}

Recall that $\Phi_n^{(\al)}\colon\cR_n^{(\al)}\to\fM_n$ is defined by~\eqref{eq:Phi}.

Given $a$ in $L^\infty([0,1))$,
denote by $\ga^{(\al)}_n(a)$
the sequence of matrices
$[\ga^{(\al)}_n(a)_\xi]_{\xi=-n+1}^\infty$,
where
$\ga^{(\al)}_n(a)_\xi\in\Mat_{\min\{n+\xi,n\}}$
and
\begin{equation}\label{eq:gamma_def}
\ga^{(\al)}_n(a)_\xi
\eqdef
\bigl[\be_{a,\al,\xi,j,k}\bigr]_{j,k=\max\{0,-\xi\}}^{n-1}.
\end{equation}

\begin{proposition}
\label{prop:radial_Toeplitz_in_poly_Bergman_converts_to_gamma}
Let $a\in L^\infty([0,1))$.
Then $T^{(\al)}_{n,\widetilde{a}}\in\cR_n^{(\al)}$
and
$\Phi_n(T^{(\al)}_{n,\widetilde{a}})
=\ga^{(\al)}_n(a)$.
\end{proposition}

\begin{proof}
Apply Propositions~\ref{prop:radial_multiplication}
and~\ref{prop:radial_Toeplitz_operator}.
\end{proof}

\subsection*{Radial Toeplitz operators in $\cA_{(n)}^2(\bD,\mu_\al)$}

\begin{proposition}\label{prop:radial_Toeplitz_eigenvalues_true_poly_Bergman}
Let $a\in L^\infty([0,1))$.
Then $T^{(\al)}_{(n),\widetilde{a}}\in\cR_{(n)}^{(\al)}$,
the operator $T^{(\al)}_{(n),\widetilde{a}}$ is diagonal
with respect to the orthonormal basis
$(\basic{\al}{p}{n-1})_{p=0}^\infty$,
and the corresponding eigenvalues
can be computed by
\begin{equation}\label{eq:radial_Toeplitz_eigenvalues_true_poly_Bergman}
\la_{a,\al,n}(p)
=\int_0^1 a(\sqrt{t})\,
\bigl(\jac_{\min\{p,n-1\}}^{(\al,|p-n+1|)}(t)\bigr)^2\,\dif{}t
\qquad(p\in\bNz).
\end{equation}
\end{proposition}

\begin{proof}
From Proposition~\ref{prop:radial_Toeplitz_operator} we get
$T^{(\al)}_{(n),\widetilde{a}}\in\cR_{(n)}^{(\al)}$.
Due to Proposition~\ref{prop:radial_multiplication}
and Theorem~\ref{thm:radial_true_polyanalytic_Bergman},
\[
\la_{a,\al,n}(p)
=(\Phi_{(n)}(T^{(\al)}_{(n),\widetilde{a}}))_p
=\langle T^{(\al)}_{(n),\widetilde{a}}\basic{\al}{p}{n-1},\basic{\al}{p}{n-1}\rangle
=\beta_{a,p-n+1,n-1,n-1}.
\qedhere
\]
\end{proof}

\section*{Acknowledgements}

The research has been supported
by CONACYT (Mexico) project ``Ciencia de Frontera 2019'', N~61517,
by IPN-SIP project 20200650
(Instituto Polit\'{e}cnico Nacional, Mexico),
and by CONACYT scholarships.
Many ideas of this paper were
inspired by talks or joint works with
Nikolai Vasilevski,
Maribel Loaiza Leyva,
Ana Mar\'{i}a Teller\'{i}a Romero,
Isidro Morales Garc\'{i}a,
and Jorge Iv\'{a}n Correo Rosas.

\medskip\noindent
Roberto Mois\'{e}s
Barrera-Castel\'{a}n\newline
Instituto Polit\'{e}cnico Nacional\newline
Escuela Superior de F\'{i}sica y Matem\'{a}ticas\newline
Apartado Postal 07730\newline
Ciudad de M\'{e}xico\newline
Mexico\newline
e-mail: rmoisesbarrera@gmail.com\newline
https://orcid.org/0000-0001-9549-3482

\bigskip\noindent
Egor A. Maximenko\newline
Instituto Polit\'{e}cnico Nacional\newline
Escuela Superior de F\'{i}sica y Matem\'{a}ticas\newline
Apartado Postal 07730\newline
Ciudad de M\'{e}xico\newline
Mexico\newline
e-mail: egormaximenko@gmail.com\newline
https://orcid.org/0000-0002-1497-4338

\bigskip\noindent
Gerardo Ramos-Vazquez\newline
Centro de Investigaci\'{o}n y de Estudios Avanzados
del Instituto Polit\'{e}cnico Nacional\newline
Departamento de Matem\'{a}ticas\newline
Apartado Postal 07360\newline
Ciudad de M\'{e}xico\newline
Mexico\newline
e-mail: ger.ramosv@gmail.com\newline
https://orcid.org/0000-0001-9363-8043


\begin{thebibliography}{35}

\bibitem{Abreu2010}
Abreu, L.D.:
On the structure of Gabor and super Gabor spaces.
\journaltitle{Monatsh. Math.}
\volumeyearpages{161}{2010}{237--253}.
\doi{10.1007/s00605-009-0177-0}

\bibitem{AbreuFeichtinger2014}
Abreu, L.D., Feichtinger, H.G.:
Function spaces of polyanalytic functions.
In: Vasil'ev, A. (eds.)
\booktitle{Harmonic and Complex Analysis and its Applications,} pp.~1--38. 
Trends in Mathematics. 
Birkh\"{a}user, Cham (2014).
\doi{10.1007/978-3-319-01806-5\_1}

\bibitem{AliBagarelloGazeau2015}
Ali, A.T., Bagarello, F., Gazeau, J.P.:
D-pseudo-bosons, complex Hermite polynomials
and integral quantization.
\journaltitle{Symmetry Integr. Geom.}
\volumeyearpages{11}{2015}{078, 23 pages}.
\doi{10.3842/SIGMA.2015.078}

\bibitem{Balk1991}
Balk, M.B.:
\booktitle{Polyanalytic Functions.}
Akad. Verlag, Berlin (1991)

\bibitem{BauerFulsche2020}
Bauer, W., Fulsche, R.: Berger-Coburn theorem, localized operators, and the Toeplitz algebra. In: Bauer, W., Duduchava, R., Grudsky, S., Kaashoek, M. (eds.) Operator Algebras, Toeplitz Operators and Related Topics, pp. 53--77. Operator Theory: Advances and Applications, vol. 279. Birkh\"{a}user, Cham  (2020).
\doi{10.1007/978-3-030-44651-2\_8}

\bibitem{BauerHerreraVasilevski2014}
Bauer, W., Herrera Ya\~nez, C., Vasilevski, N.:
Eigenvalue characterization of radial operators
on weighted Bergman spaces over the unit ball.
\journaltitle{Integral Equ. Oper. Theory}
\volumeyearpages{78}{2014}{1--30}.
\doi{10.1007/s00020-013-2101-1}

\bibitem{BergerCoburn1986}
Berger, C.A., Coburn, L.A.:
Toeplitz operators and quantum mechanics.
\journaltitle{J. Funct. Anal.}
\volumeyearpages{68}{1986}{273--299}.
\doi{10.1016/0022-1236(86)90099-6}

\bibitem{DawsonOlafssonQuiroga2015}
Dawson, M., \'{O}lafsson, G., Quiroga-Barranco, R.:
Commuting Toeplitz operators on bounded symmetric domains 
and multiplicity-free restrictions of holomorphic discrete series.
\journaltitle{J. Funct. Anal.}
\volumeyearpages{268}{2015}{1711--1732}.
\doi{10.1016/j.jfa.2014.12.002}

\bibitem{Dzhuraev1992}
Dzhuraev, A.:
\booktitle{Methods of Singular Integral Equations.}
Longman Scientific \& Technical, Harlow
(1992)


\bibitem{Englis2006}
Engli\v{s}, M.:
Berezin and Berezin-Toeplitz quantizations
for general function spaces.
\journaltitle{Rev. Mat. Complut.}
\volumeyearpages{19}{2006}{385--430}.
\myurl{http://eudml.org/doc/41908}


\bibitem{GrudskyQuirogaVasilevski2006}
Grudsky, S., Quiroga-Barranco, R., Vasilevski, N.:
Commutative C*-algebras of Toeplitz operators
and quantization on the unit disk.
\journaltitle{J. Funct. Anal.}
\volumeyearpages{234}{2006}{1--44}.
\doi{10.1016/j.jfa.2005.11.015}

\bibitem{GrudskyMaximenkoVasilevski2013}
Grudsky, S.M., Maximenko, E.A., Vasilevski, N.L.:
Radial Toeplitz operators on the unit ball and slowly oscillating sequences.
\journaltitle{Commun. Math. Anal.}
\volumeissueyearpages{14}{2}{2013}{77--94}.
\myurl{https://projecteuclid.org/euclid.cma/1356039033}

\bibitem{GrudskyVasilevski2002}
Grudsky, S., Vasilevski, N.:
Toeplitz operators on the Fock space:
Radial component effects.
\journaltitle{Integral Equ. Oper. Theory}
\volumeyearpages{44}{2002}{10--37}.
\doi{10.1007/BF01197858}

\bibitem{HachadiYoussfi2019}
Hachadi, H., Youssfi, E.H.:
The polyanalytic reproducing kernels.
\journaltitle{Complex Anal. Oper. Theory},
\volumeyearpages{13}{2019}{3457--3478}.
\doi{10.1007/s11785-019-00956-5}

\bibitem{HaimiHedenmalm2013}
Haimi, A., Hedenmalm, H.:
The polyanalytic Ginibre ensembles.
\journaltitle{J. Stat. Phys.}
\volumeyearpages{153}{2013}{10--47},
\doi{10.1007/s10955-013-0813-x}.

\bibitem{HerreraVasilevskiMaximenko2015}
Herrera Ya\~{n}ez, C., Vasilevski, N.,
Maximenko, E.A.:
Radial Toeplitz operators revisited:
Discretization of the vertical case.
\journaltitle{Integral Equ. Oper. Theory}
\volumeyearpages{83}{2015}{49--60}.
\doi{10.1007/s00020-014-2213-2}

\bibitem{Hutnik2008}
Hutn\'{i}k, O.:
On the structure of the space of wavelet transforms.
\journaltitle{C. R. Acad. Sci. Paris, Ser. I}
\volumeyearpages{346}{2008}{649--652}.
\doi{10.1016/j.crma.2008.04.013}

\bibitem{Hutnik2010}
Hutn\'{i}k, O.:
A note on wavelet subspaces.
\journaltitle{Monatsh. Math.}
\volumeyearpages{160}{2010}{59--72}.
\doi{10.1007/s00605-008-0084-9}

\bibitem{HutnikMaximenkoMiskova2016}
Hutn\'{i}k, O., Maximenko, E., Mi\v{s}kov\'{a}, A.:
Toeplitz localization operators:
spectral functions density.
\journaltitle{Complex Anal. Oper. Theory}
\volumeyearpages{10}{2016}{1757--1774}.
\doi{10.1007/s11785-016-0564-1}

\bibitem{HutnikHutnikova2015}
Hutn\'{i}k, O., Hutn\'{i}kov\'{a}, M.:
Toeplitz operators on poly-analytic spaces via time-scale analysis.
\journaltitle{Oper. Matrices}
\volumeyearpages{8}{2015}{1107--1129}.
\doi{10.7153/oam-08-62}

\bibitem{Koornwinder1975}
Koornwinder T.H.:
Two-variable analogues of the classical orthogonal polynomials. In Askey, R.A. (ed.): Theory and Application of Special Functions, pp. 435--495.
Academic Press, New York (1975)

\bibitem{KorenblumZhu1995}
Korenblum, B., Zhu, K.:
An application of Tauberian theorems to Toeplitz operators.
\journaltitle{J. Oper. Th.}
\volumeyearpages{33}{1995}{353--361}.
\myurl{https://www.jstor.org/stable/24714916}

\bibitem{Koshelev1977}
Koshelev, A.D.:
On the kernel function of the Hilbert space
of functions polyanalytic in a disc.
\journaltitle{Dokl. Akad. Nauk SSSR}
\volumeyearpages{232}{1977}{277--279}


\bibitem{LoaizaLozano2014}
Loaiza, M., Lozano, C.:
On Toeplitz operators on the weighted
harmonic Bergman space on the upper half-plane.
\journaltitle{Complex Anal. Oper. Theory}
\volumeyearpages{9}{2014}{139--165}.
\doi{10.1007/s11785-014-0388-9}

\bibitem{LoaizaRamirez2017}
Loaiza, M., Ram\'{i}rez-Ortega, J.:
Toeplitz operators with homogeneous symbols acting on the poly-Bergman spaces of the upper half-plane.
\journaltitle{Integral Equ. Oper. Theory}
\volumeyearpages{87}{2017}{391--410}.
\doi{10.1007/s00020-017-2350-5}

\bibitem{MaximenkoTelleriaRomero2020}
Maximenko, E.A., Teller\'{i}a-Romero, A.M.:
Radial operators in polyanalytic Bargmann--Segal--Fock spaces.
In: Bauer, W., Duduchava, R., Grudsky, S., Kaashoek, M. (eds.) Operator Algebras, Toeplitz Operators and Related Topics, pp. 277--305.
Operator Theory: Advances and Applications, vol 279.
Birkh\"{a}user, Cham (2020).
\doi{10.1007/978-3-030-44651-2\_18}

\bibitem{Pessoa2014}
Pessoa, L.V.:
Planar Beurling transform and Bergman type spaces.
\journaltitle{Complex Anal. Oper. Theory}
\volumeyearpages{8}{2014}{359--381}.
\doi{10.1007/s11785-012-0268-0}

\bibitem{Quiroga2016}
Quiroga-Barranco, R.:
Separately radial and radial Toeplitz operators
on the unit ball and representation theory.
\journaltitle{Bol. Soc. Mat. Mex.}
\volumeyearpages{22}{2016}{605--623}.
\doi{10.1007/s40590-016-0111-0}

\bibitem{Ramazanov1999}
Ramazanov, A.K.:
Representation of the space of polyanalytic functions as a direct sum of orthogonal subspaces.
Application to rational approximations.
\journaltitle{Math. Notes}
\volumeyearpages{66}{1999}{613--627}.
\doi{10.1007/BF02674203}

\bibitem{Ramazanov2002}
Ramazanov, A.K.:
On the structure of spaces of polyanalytic functions.
\journaltitle{Math. Notes}
\volumeyearpages{72}{2002}{692--704}.
\doi{10.1023/A:1021469308636}

\bibitem{RamirezSanchez2015}
Ram\'{i}rez Ortega, J., S\'{a}nchez-Nungaray, A.:
Toeplitz operators with vertical symbols
acting on the poly-Bergman spaces
of the upper half-plane.
\journaltitle{Complex Anal. Oper. Theory}
\volumeyearpages{9}{2015}{1801--1817}.
\doi{10.1007/s11785-015-0469-4}

\bibitem{RamirezRamirezSanchez2019}
Ram\'{i}rez Ortega, J.,
Ram\'{i}rez Mora, M.R.,
S\'{a}nchez Nungaray, A.:
Toeplitz operators with vertical symbols acting on the poly-Bergman spaces of the upper half-plane. II.
\journaltitle{Complex Anal. Oper. Theory}
\volumeyearpages{13}{2019}{2443--2462}.
\doi{10.1007/s11785-019-00908-z}

\bibitem{RozenblumVasilevski2019}
Rozenblum, G.; Vasilevski, N.:
Toeplitz operators in polyanalytic Bergman type spaces.
In: Kuchment, P., Semenov, E. (eds.)
Functional analysis and geometry:
Selim Grigorievich Krein centennial,
273--290.
Contemp. Math., vol.~733,
Amer. Math. Soc., Providence, RI (2019).
\doi{10.1090/conm/733/14747}


\bibitem{Sakai1971}
Sakai, S.:
$C^\ast$-algebras and $W^\ast$-algebras,
Springer-Verlag, Berlin, Heilderberg, New York (1971)

\bibitem{SanchezGonzalezLopezArroyo2018}
S\'anchez-Nungaray, A., Gonz\'alez-Flores, C., L\'opez-Mart\'inez, R.R., Arroyo-Neri, J.L.:
Toeplitz operators with horizontal symbols
acting on the poly-Fock spaces.
\journaltitle{J. Funct. Spaces}
\volumeyearpages{2018}{2018}{Article ID 8031259, 8 pages}.
\doi{10.1155/2018/8031259}


\bibitem{Stroethoff1997}
Stroethoff, K.:
The Berezin transform and operators
on spaces of analytic functions.
\journaltitle{Banach Center Publ.}
\volumeyearpages{38}{1997}{361--380}.
\doi{10.4064/-38-1-361-380}

\bibitem{Suarez2008}
Su\'{a}rez, D.:
The eigenvalues of limits of radial Toeplitz operators.
\journaltitle{Bull. Lond. Math. Soc.}
\volumeyearpages{40}{2008}{631--641}.
\doi{10.1112/blms/bdn042}

\bibitem{Szego1975}
Szeg\H{o}, G.:
Orthogonal Polynomials, 4th ed.
Amer. Math. Soc., Providence, R.I. (1975)

\bibitem{Trofymenko2018}
Trofymenko, O.D.:
Convolution equations and mean-value theorems for solutions of linear elliptic equations with constant coefficients in the complex plane.
J. Math. Sciences
\volumeyearpages{229}{2018}{96--107}.
\doi{10.1007/s10958-018-3664-9}

\bibitem{Vasilevski1999}
Vasilevski, N.L.:
On the structure of Bergman and poly-Bergman spaces.
\journaltitle{Integral Equ. Oper. Theory}
\volumeyearpages{33}{1999}{471--488}.
\doi{https://doi.org/10.1007/BF01291838}

\bibitem{Vasilevski2000}
Vasilevski, N.L.:
Poly-Fock spaces.
In: Adamyan, V.M., et al. (eds.) Differential Operators and Related Topics,
371--386.
Operator Theory: Advances and Applications, vol. 117,
Birkh\"{a}user, Basel (2000).
\doi{10.1007/978-3-0348-8403-7\_28}

\bibitem{Vasilevski2008book}
Vasilevski, N.L.:
\booktitle{Commutative Algebras of Toeplitz Operators on the Bergman Space}.
Birkh\"{a}user, Basel, Boston, Berlin  (2008).
\doi{10.1007/978-3-7643-8726-6}

\bibitem{Wunsche2005}
W\"{u}nsche, A.:
Generalized Zernike or disc polynomials.
J. Comput. Appl. Math.
\volumeissueyearpages{174}{1}{2005}{135-163}.
\doi{10.1016/j.cam.2004.04.004}

\bibitem{Xia2015}
Xia, J.:
Localization and the Toeplitz algebra on the Bergman space.
J. Funct. Anal.
\volumeyearpages{269}{2015}{781--814}.
\doi{10.1016/j.jfa.2015.04.011}


\bibitem{Zhu2007book}
Zhu, K.:
\booktitle{Operator Theory in Function Spaces. 2nd. ed.}
Amer. Math. Soc., Providence, R.I. (2007). 
\doi{10.1090/surv/138}

\bibitem{Zorboska2003}
Zorboska, N.:
The Berezin transform and radial operators.
\journaltitle{Proc. Am. Math. Soc.}
\volumeyearpages{131}{2003}{793--800}.
\myurl{https://www.jstor.org/stable/1194482}

\end{thebibliography}
\end{document}